\documentclass{article}
\usepackage[utf8]{inputenc}

\title{\dwifob: A Dynamically Weighted Inertial Forward--Backward Algorithm for Monotone Inclusions}

\author{Hamed Sadeghi\thanks{Email: \{\href{mailto:hamed.sadeghi@control.lth.se}{hamed.sadeghi}, \href{mailto:sebastian.banert@control.lth.se}{sebastian.banert}, \href{mailto:pontus.giselsson@control.lth.se}{pontus.giselsson}\}@control.lth.se. Affiliation: Department of Automatic Control, Lund University, Lund, Sweden.}
   \and Sebastian Banert\footnotemark[1]
   \and Pontus Giselsson\footnotemark[1]}
\date{}

\usepackage{microtype}
\usepackage{amsmath,amssymb}
\usepackage{amsthm}
\usepackage{bold-extra}
\usepackage{float}
\usepackage{subcaption}
\usepackage{caption}
\usepackage{enumitem}
\usepackage{algorithm}
\usepackage{algpseudocode}
\usepackage{comment}
\usepackage{lipsum}
\usepackage{amsfonts}
\usepackage{graphicx}
\usepackage{epstopdf}
\usepackage{csvsimple}
\usepackage{tikz}
\usetikzlibrary{matrix,arrows,positioning,shadows,intersections,calc}
\usepackage{pgfplots}
\usepackage{pgfplotstable}
\pgfplotsset{compat=newest}
\usepackage{hyperref}
\usepackage[capitalise]{cleveref}

\usepackage{xcolor}
\definecolor{redd}{rgb}{0.8,0.22,0.0} %red
\definecolor{bluee}{rgb}{0.23,0.37,0.8} %blue
\definecolor{brownn}{rgb}{0.72,0.53,0.04} %brown
\definecolor{greenn}{rgb}{0.0,0.55,0.0}   %green
\definecolor{blackk}{rgb}{0.0,0.0,0.0} %black
\definecolor{grayy}{rgb}{0.6,0.6,0.6} %gray

\usepackage[mathscr]{eucal}

\usepackage[giveninits=true,style = numeric,backend=biber,sorting=nyvt]{biblatex}
\DeclareNameAlias{default}{family-given}
\newcommand{\PrimS}{\ensuremath{\mathcal{H}}}

\newcommand{\nat}{\ensuremath{\mathbb{N}}}

\newcommand{\Id}{\ensuremath{\mathop{\mathrm{Id}}}}
\newcommand{\p}[1]{\ensuremath{\mathord{\left( #1 \right)}}}
\newcommand{\norm}[1]{\ensuremath{\mathord{\left\Vert #1 \right\Vert}}}
\newcommand{\set}[1]{\ensuremath{\mathord{\left\lbrace #1 \right\rbrace}}}

\newcommand{\inpr}[2]{\ensuremath{\mathord{\left\langle #1, #2 \right\rangle}}}

\newcommand{\normsq}[1]{\left\Vert #1 \right\Vert^2}

\newcommand{\zer}[1]{\operatorname{\mathrm{zer}}(#1)}
\newcommand{\reals}{\mathbb{R}}

\newcommand{\posop}{\mathcal{M}(\PrimS)}
\newcommand{\lp}{\left(}
\newcommand{\rp}{\right)}
\newcommand{\xs}{x^\star}

\newcommand{\gra}[1]{\operatorname{gra}(#1)}
\newcommand{\seq}[2]{$(#1_#2)_{#2\in\mathbb{N}}$}

\newcommand{\dwifob}{\textsc{Dwifob}}
\newcommand{\defeq}{\ensuremath{\mathrel{\mathop:}=}}
\newcommand{\ccp}{\ensuremath{C_{\mathrm{CP}}}}
\newcommand{\calg}{\ensuremath{C_{\mathrm{alg}}}}

\theoremstyle{plain}
\newtheorem{theorem}{Theorem}

\theoremstyle{plain}

\theoremstyle{plain}

\theoremstyle{definition}

\theoremstyle{definition}

\theoremstyle{plain}
\newtheorem{assumption}{Assumption}
\crefname{assumption}{Assumption}{Assumptions}
\Crefname{assumption}{Assumption}{Assumptions}

\theoremstyle{plain}
\newtheorem{corollary}{Corollary}

\theoremstyle{plain}
\newtheorem{remark}{Remark}

\begin{document}
\maketitle

\begin{abstract}
We propose a novel \emph{dynamically weighted inertial forward--backward} algorithm (\dwifob) for solving structured monotone inclusion problems. The scheme exploits the globally convergent forward--backward algorithm with deviations in \cite{nofob-increments} as the basis and combines it with the extrapolation technique used in Anderson acceleration to improve local convergence. 
We also present a globally convergent primal--dual variant of \dwifob{} and numerically compare its performance to the primal--dual method of Chambolle--Pock and a Tikhonov regularized version of Anderson acceleration applied to the same problem. In all our numerical evaluations, the primal--dual variant of \dwifob{} outperforms the Chambolle--Pock algorithm. Moreover, our numerical experiments suggest that our proposed method is much more robust than the regularized Anderson acceleration, which can fail to converge and be sensitive to algorithm parameters. These numerical experiments highlight that our method performs very well while still being robust and reliable.

\end{abstract}


\section{Introduction}

We consider structured monotone inclusion problems of the form
\begin{equation}\label{eq:main_inclusion}
    0\in Ax+Cx,
\end{equation}
where $A:\PrimS\rightarrow2^\PrimS$ is a maximally monotone operator, $C:\PrimS\rightarrow\PrimS$ is a cocoercive operator, and $\PrimS$ is a real Hilbert space. This fundamental problem emerges in many areas such as optimization \cite{Eckstein1989SplittingMF,Raguet_2015} and variational analysis \cite{Attouch_coupling,Chen_1997,Tseng_2000}.

\emph{Forward--backward} (FB) splitting \cite{Bruck1975AnIS,Lions_1979,Passty_1979} has been widely used to solve structured monotone inclusions of the form \eqref{eq:main_inclusion}. The FB splitting method is given by 
\begin{align*}
    x_{n+1}=(\Id + \gamma_nA)^{-1}\circ(\Id-\gamma_nC)(x_n),
\end{align*}
where $\gamma_n > 0$ is a step-size parameter. It involves evaluating the operator $C$ in a forward (explicit) step, followed by computing the resolvent of the operator $A$ in a backward (implicit) step. The FB splitting has many well-known special instances, such as the gradient method, the proximal point algorithm \cite{rockafellar1976monotone}, and the proximal-gradient method \cite{Combettes2011ProximalSM}.

The {inertial proximal point} algorithm in \cite{Alvarez_2000,Alvarez_2001} improves convergence by exploiting previous information in a momentum term. By incorporating an additional cocoercive operator to the inertial proximal point algorithm, several variations of inertial FB algorithms have been proposed to solve monotone inclusions \cite{attouch2020convergence,Cholamjiak_2018,lorenz2015inertial}. These algorithms provide enhanced performance, but are limited to FB splitting algorithms.

\emph{Anderson acceleration} \cite{Anderson1965IterativePF} is an acceleration scheme that is aimed at expediting the convergence of fixed-point iterations including the FB algorithm. This algorithm was originally developed to solve nonlinear integral equations and was later used to solve fixed-point problems \cite{fang2009two,walker2011anderson}. Lately, Anderson acceleration has gained considerable attention in the optimization community \cite{he2021solve,ouyang2020anderson,sad2021Hyb,scieur2020regularized,zhang2020globally}. 

Local convergence of Anderson acceleration has been studied recently. For instance, the authors of \cite{conv_AA_Kelley} showed that Anderson acceleration, if applied to a contractive fixed-point map, exhibits linear convergence provided that the coefficients in the linear combination remain bounded. Along the same line, it was shown in \cite{AA_Evans} that applying Anderson acceleration to a linearly convergent fixed-point iteration improves the convergence rate in the vicinity of a fixed point. Despite recent studies that investigate local convergence properties of Anderson acceleration, yet, to the best of our knowledge, no global convergence result for Anderson acceleration (and its regularized variants) has been reported in the literature.

Recently, the FB algorithm with deviations was proposed in \cite{nofob-increments} to solve the inclusion problem \eqref{eq:main_inclusion}. This algorithm uses two auxiliary terms---called \emph{deviations}---which are added to the iterates in order to define extrapolated iterates.  
The algorithm uses a safeguarding \emph{norm condition} in the form of an iteration-dependent constraint on the norm of the deviations that has to be satisfied at each iteration in order to guarantee convergence. 
As long as this norm constraint is satisfied, the deviations can be chosen freely and point in any direction. In \cite{nofob-increments}, one suggestion is to define the deviations along the momentum direction as $a_n(x_n-x_{n-1})$, which gives an inertial-type method. An upper bound to the momentum coefficient $a_n$ is directly obtained by the norm condition.

In this work, inspired by the extrapolation technique of Anderson acceleration, we propose a method to generate the deviation vectors of \cite{nofob-increments} by linearly combining multiple momentum terms. The aim is to construct a version of FB splitting that exhibits fast local convergence while maintaining global convergence of the algorithm, thanks to the norm condition. This is in contrast to Anderson acceleration and its regularized variants \cite{scieur2020regularized,shi2019regularized} that are only locally convergent. We call our proposed algorithm \emph{dynamically weighted inertial forward--backward} method (\dwifob).

The notion of safeguarding has been used also in other works to ensure global convergence of nonlinear acceleration algorithms \cite{giselsson2016line,sad2021Hyb,themelis2019supermann,zhang2020globally}. These are hybrid methods that can select between a basic globally convergent and a locally fast converging method, as decided by a safeguarding condition in every iteration. Although having the same objective of achieving global convergence and fast local convergence, these safeguarding conditions are completely different compared to what we use in \dwifob.

Besides the \dwifob{} scheme itself, we also propose a primal--dual version of the \dwifob{} scheme which is derived by a direct translation of the \dwifob{} algorithm into a primal--dual framework. We have compared the primal--dual \dwifob{} algorithm with the Chambolle--Pock algorithm in numerical experiments, which show a significant advantage of our proposed method in both convergence rate and overall computational cost. Moreover, our numerical evaluations show that regularized Anderson acceleration, in addition to being only locally convergent, is very sensitive to variations in the choice of parameters, while \dwifob{} is more robust to parameter selection with the significant added benefit of having global convergence guarantees. The aforementioned robustness and global convergence property along with fast local convergence make the \dwifob{} algorithm well-performing and  reliable.

The paper is outlined as follows. In \cref{sec:basic_algs}, after presenting the notations and stating the problem under consideration, we review two algorithms that our algorithm is built upon. \Cref{sec:main_alg} describes our proposed \dwifob{} algorithm and \cref{sec:pd-DWIFOB} extends the \dwifob{} algorithm to the primal--dual setting and suggests a novel algorithm in this framework. Numerical evaluations are provided in \cref{sec:num-exp} and concluding remarks are presented in \cref{sec:conclusion}.

\section{Problem statement and preliminaries}\label{sec:basic_algs}

In this section, we present our notation and state the monotone inclusion problem and the associated assumptions. We then briefly review two methods \cite{nofob-increments,walker2011anderson} that can be used to solve the problem at hand. These methods come with their own sets of weaknesses and strengths. Our proposed method combines these two methods to benefit from their individual strengths and avoid their drawbacks.

\subsection{Notation}

Throughout the paper, $\reals$ and $\reals^d$ indicate the sets of real numbers and $d$-dimensional real column vectors respectively. Additionally, $\PrimS$ and $\mathcal{K}$ denote real Hilbert spaces that are equipped with inner products $\inpr{\cdot}{\cdot}$ and induced norms $\norm{\cdot}=\sqrt{\inpr{\cdot}{\cdot}}$. A linear, bounded, self-adjoint operator $M\colon \PrimS\to \PrimS$ is said to be \emph{strongly positive} if there exists $\rho > 0$ such that $\inpr{x}{Mx}\geq\rho\normsq{x}$ for all $x\in\PrimS$. We denote the set of such operators $\mathcal{M}(\PrimS)$. For $M\in\mathcal{M}(\PrimS)$, the \emph{$M$-induced inner product} and \emph{norm} are defined by $\inpr{x}{y}_M=\inpr{x}{My}$ and $\norm{x}_M = \sqrt{\inpr{x}{Mx}}$ ($x, y \in \PrimS$), respectively.

By $2^\PrimS$, we denote the \emph{power set} of $\PrimS$. A map $A:\PrimS\rightarrow2^{\PrimS}$ is characterized by its \emph{graph} $\gra{A} = \{(x,u)\in\PrimS\times\PrimS : u\in Ax\}$. An operator $A:\PrimS\rightarrow2^{\PrimS}$ is \emph{monotone}, if $\inpr{u-v}{x-y}\geq0$ for all $(x,u),(y,v)\in\gra{A}$. A monotone operator $A$ is \emph{maximally monotone} if there exists no monotone operator $B:\PrimS\rightarrow 2^\PrimS$ such that $\gra{B}$ properly contains $\gra{A}$. The \emph{zero-set} of the operator $A$ is defined as $\zer{A} \defeq \{x\in\PrimS: 0\in Ax\}.$

For $\beta > 0$, a single-valued operator $T\colon\PrimS\to\PrimS$ is said to be \emph{$\tfrac{1}{\beta}$-cocoercive} with respect to \ $\norm{\cdot}_M$ with $M\in\posop$ if
\begin{equation*}
    \inpr{Tx -Ty}{x-y}\geq\tfrac{1}{\beta}\norm{Tx-Ty}_{M^{-1}}^2\qquad(\forall x,y\in\PrimS).
\end{equation*}

\subsection{Problem statement}\label{sec:FB-problem-formulation}
We consider structured monotone inclusion problems of the form
\begin{align} \label{eq:monotone_inclusion}
    0 \in Ax + Cx,
\end{align}
that satisfy the following assumption.

\begin{assumption}
\label{assum:monotone_inclusion}
Assume that
\renewcommand{\labelenumi}{\emph{(\roman{enumi})}}
\begin{enumerate}
    \item $A: \PrimS \rightarrow 2^{\PrimS}$ is maximally monotone.
    \item $C:\PrimS \rightarrow \PrimS$ is $\tfrac{1}{\beta}$-cocoercive with respect to $\norm{\cdot}_M$ with $M\in\mathcal{M}(\PrimS)$.
    \item The solution set $\zer{A+C}$ is nonempty.
\end{enumerate}
\end{assumption}

This assumption implies that the operator $A+C$ is maximally monotone \cite[Corollary 25.5]{bauschke2017convex}.

\subsection{Forward--backward splitting with deviations}
The FB algorithm with deviations is an extension of the standard FB algorithm and was introduced recently in \cite{nofob-increments}. In its most general form, two additive terms---called deviations---are added to the basic FB method to form extrapolations to the iterate. The algorithm uses the extrapolated points in the evaluation of the forward and the backward steps. If the deviations are chosen wisely, this can exhibit an improved convergence compared to standard FB splitting. \Cref{alg:PFOB} presents an instance of the FB algorithm with only one deviation vector.

\begin{algorithm}
	\caption{}
	\begin{algorithmic}[1]
	    \State \textbf{Input:} $x_0\in \PrimS$; and the sequences \seq{\gamma}{n}, \seq{\lambda}{n}, and \seq{\zeta}{n} according to \cref{assum:parameters}; and the metric $\norm{\cdot}_M$ with $M\in\posop$.
	    \State \textbf{set:} $y_0=x_0$ and $u_0=0$.
		\For {$n=0,1,2,\ldots$}
		    \State $p_n = (M+\gamma_nA)^{-1}\circ(M - \gamma_n C)(y_n)$
		    \State $x_{n+1} = x_n + \lambda_n(p_n - y_n)$
		    \State\label{step:pfob-norm-condition}choose $u_{n+1}$ such that
		    \begin{equation*}
            \begin{aligned}
                \norm{u_{n+1}}_{M}^2 \leq \zeta_{n}^2\tfrac{\lambda_n(4-2\lambda_n-\gamma_n\beta)(4-2\lambda_{n+1}-\gamma_{n+1}\beta)}{4\lambda_{n+1}}\norm{p_n-x_n+\tfrac{2\lambda_n+\gamma_n\beta-2}{4-2\lambda_n-\gamma_n\beta}u_n}_M^2
            \end{aligned}
            \end{equation*}\label{alg-line:alg1-bound_on_deviations}
		    \State$y_{n+1} = x_{n+1} + u_{n+1}$
		\EndFor
	\end{algorithmic}
\label{alg:PFOB}
\end{algorithm}

To ensure convergence of \cref{alg:PFOB}, the deviation $u_{n+1}$ must satisfy the iteration-dependent norm bound in \cref{alg-line:alg1-bound_on_deviations} at each iteration \cite{nofob-increments}. This bound is referred to as a \emph{norm condition}. The requirements on the parameters $\lambda_n$, $\gamma_n$, and $\zeta_n$ are collected in \cref{assum:parameters}.

\begin{assumption}\label{assum:parameters}
Choose $\epsilon\in\p{0, \min\set{1, \tfrac{4}{3+\beta}}}$, and assume that, for all $n\in\nat$, the following hold: 
\renewcommand{\labelenumi}{{(\roman{enumi})}}
\begin{enumerate}[noitemsep,ref=\cref{assum:parameters}~\theenumi]
\item $0 \leq \zeta_n \leq 1 - \epsilon$; \label{itm:assump-param-i}
\item $\epsilon \leq \gamma_n \leq \frac{4 - 3\epsilon}{\beta}$; and \label{itm:assump-param-ii} 
\item $\epsilon \leq \lambda_n \leq 2 - \frac{\gamma_n \beta}{2} - \frac{\epsilon}{2}$. \label{itm:assump-param-iii}
\end{enumerate}

\end{assumption}

The following result, which is adopted from \cite{nofob-increments}, provides a convergence guarantee for the iterates that are obtained from \cref{alg:PFOB}.

\begin{theorem}\label{thm:main}
Consider the monotone inclusion problem \eqref{eq:monotone_inclusion} and suppose that \cref{assum:monotone_inclusion,assum:parameters} hold. Let $(x_n)_{n\in\nat}$ be the sequence generated by \cref{alg:PFOB}. Then, the sequence $(x_n)_{n\in\nat}$ converges weakly to a point in $\zer{A+C}$.
\end{theorem}
\begin{proof}
In the FB splitting with deviations \cite[Algorithm~1]{nofob-increments}, set $z_n = y_n$. This gives the relation
\begin{align*}
    v_{n} = \tfrac{2-\gamma_n\beta}{2-\lambda_n\gamma_n\beta}u_n
\end{align*}
between $u_n$ and $v_n$, which yields  \cref{alg:PFOB}. Therefore, \cref{alg:PFOB} is an instance of the FB splitting algorithm with deviations; consequently, \cref{thm:main} is a direct consequence of \cite[Theorem~1]{nofob-increments}.
\end{proof}

There is a great flexibility in the choice of deviation vector $u_{n+1}$. This flexibility has not been fully explored in \cite[Section~6]{nofob-increments}, where only a simple momentum direction has been considered. Our proposed method is an instance of \cref{alg:PFOB} from \cite{nofob-increments}, where the deviations are chosen based on ideas from the extrapolation step of Anderson acceleration with the goal of improving local performance while benefiting from the global convergence properties of \cref{alg:PFOB}.

%%%%%%%%%%%%%%%%%%%%%%%%%%%%%%%%%%%%%%%%%%%
\subsection{Regularized Anderson acceleration}
\label{subsec:Anderson}
Consider the following fixed-point problem 
\begin{equation} \label{eq:fixedPointIteration}
    \text{find~} x\in\PrimS \text{~such that~} x = T(x),
\end{equation}
where $T\colon \PrimS \to \PrimS$ is a nonexpansive mapping. One way to solve this problem is to use \emph{Anderson acceleration} \cite{Anderson1965IterativePF, walker2011anderson}. Anderson acceleration is easy to implement and often improves the convergence of fixed-point iterations, particularly in their terminal phase of convergence, i.e., when close to a solution. However, Anderson acceleration (in its original form \cite{Anderson1965IterativePF, walker2011anderson}) suffers from numerical instability. This issue can, to some extent, be addressed by adding a Tikhonov regularization term to its inner least-squares problem. A regularized formulation of Anderson acceleration is given in \cref{alg:AA} \cite{scieur2020regularized,shi2019regularized}. In spite of their popularity and benefits, there are not yet any global convergence results for the pure Anderson acceleration or its regularized variant, to the best of our knowledge.

\begin{algorithm}
	\caption{Regularized Anderson acceleration}
	\begin{algorithmic}[1]
	    \State \textbf{Input:} $y_0\in\PrimS$; $m\geq1$; and the  regularization parameter $\xi$.
		\For {$n=0,1,2,\ldots$}
		    \State $m_n = \min\{m,n\}$
		    \State $x_{n} = T(y_{n})$
		    \State find $\alpha^{(n)} = (\alpha_0^{(n)},\ldots,\alpha_{m_n}^{(n)})$ that solves
		        \begin{align*}              &\underset{~~\alpha^{(n)}\in\reals^{m_n+1}}{\mathrm{minimize}} ~~~ \norm{\mathscr{R}_n\alpha^{(n)}}_2^2 + \xi\norm{\mathscr{R}_n^T\mathscr{R}_n}_F\norm{\alpha^{(n)}}_2^2 \\
		        &\mathrm{~~subject ~ to ~~~~}  \mathbf{1}^T\alpha^{(n)}=1
		        \end{align*}
		        \hspace{5.2mm}where $\mathscr{R}_n = (r_{n-m_n},\ldots,r_n)$ and $r_{j} = y_{j} - x_{j}$ for $j\in\{n-m_n,\ldots,n\}$ \label{step:aa-least-squares-problem}
		    \State  $y_{n+1}=\sum_{i=0}^{m_n}\alpha_{i}^{(n)}x_{n-m_n+i}$\label{step:aa-extrapolated-step}
		\EndFor
	\end{algorithmic}
\label{alg:AA}
\end{algorithm}

Anderson acceleration is retrieved from \cref{alg:AA} by setting $\xi=0$. The original formulation of Anderson acceleration \cite{Anderson1965IterativePF} is more general as it allows for the following damped (mixed) step to be taken
\begin{align*}
	y_{n+1} = \mu_n\sum_{i=0}^{m_n}\alpha_{i}^{(n)}x_{n-m_n+i} + (1-\mu_n)\sum_{i=0}^{m_n}\alpha_{i}^{(n)}y_{n-m_n+i},
\end{align*}
instead of \cref{step:aa-extrapolated-step}, in which $\mu_n>0$ is the damping (mixing) parameter. In this work, we consider the regularized variant of Anderson acceleration, given in \cref{alg:AA}, and refer to it as RAA.

\begin{remark}{\label{rem:QN-AA}}
Anderson acceleration (\cref{alg:AA} with $\xi=0$) can be viewed as a  quasi-Newton method \cite{eyert1996comparative,fang2009two, walker2011anderson,zhang2020globally}. To see this, first observe that the inner optimization problem of Anderson acceleration can be written as the following unconstrained least-squares problem
\begin{align}\label{eq:inner_opt_AA_unconstrained}
    \underset{\omega^{(n)}\in\reals^{m_n}}{\mathrm{minimize}} ~~~ \norm{r_n- \Delta\mathscr{R}_n\omega^{(n)}}_2,
\end{align}
where $\Delta\mathscr{R}_n = (r_{n-m_n+1}-r_{n-m_n},\ldots,r_{n}-r_{n-1})$ and $\omega^{(n)} = (\omega_0^{(n)},\ldots,\omega_{m_n-1}^{(n)})$ with $\omega_i^{(n)} = \sum_{j=0}^i\alpha_j^{(n)}$ for $i\in\{0,\ldots,m_{n}-1\}$. Then, defining $\Delta\mathscr{Y}_n = (y_{n-m_n+1}-y_{n-m_n},\ldots,y_{n}-y_{n-1})$, the extrapolation step of AA can be cast as
\begin{align*}
    y_{n+1}= y_n -  G_n r_n
\end{align*}
where $G_n = \Id+(\Delta\mathscr{Y}_n-\Delta\mathscr{R}_n)(\Delta\mathscr{R}_n^T\Delta\mathscr{R}_n)^{-1}\Delta\mathscr{R}_n^T$. In this framework, Anderson acceleration can be seen a quasi-Newton method where $G_n$ is an approximate inverse Jacobian of $x-T(x)$ that minimizes $\Vert G_n-I\Vert_F$ subject to the inverse multi-secant condition $G_n\Delta\mathscr{Y}_n=\Delta\mathscr{R}_n$.
\end{remark}

\section{Dynamically weighted inertial FB scheme}\label{sec:main_alg}
In this section, we present a dynamically weighted inertial forward--backward (\dwifob{}) scheme to solve the problem introduced in \cref{sec:FB-problem-formulation}. It is based on \cref{alg:PFOB} with a choice of deviation vectors inspired by RAA (\cref{alg:AA}).

The \dwifob{} scheme exploits a history of search directions similar to RAA to find a deviation vector, and it uses the norm condition in~\cref{step:pfob-norm-condition} of \cref{alg:PFOB} to bound the norm of the deviation. This results in an algorithm that addresses the drawbacks of \cref{alg:PFOB} (slow local convergence) and RAA (no global convergence guarantee) and benefits from their favorable properties; namely, global convergence of \cref{alg:PFOB} and the often fast local convergence of RAA.

\begin{algorithm}[h]
	\caption{\dwifob{}}
	\begin{algorithmic}[1]
	    \State \textbf{Input:} $x_0\in\PrimS$; $m\geq1$;  the sequences \seq{\lambda}{n}, \seq{\gamma}{n}, and \seq{\zeta}{n} as defined in \cref{assum:parameters}; the regularization parameter $\xi$; the metric $\norm{\cdot}_M$ with $M\in\posop$; and $\varepsilon\geq0$.
	    \State \textbf{set} $y_0 = x_0$  and $u_0=0$.
		\For {$n=0,1,2,\ldots$}
		    \State $m_n = \min(m,n)$
		    \State $p_{n}=(M+\gamma_nA)^{-1}\circ (M-\gamma_nC)y_{n}$
		    \State $x_{n+1}=x_{n}+\lambda_{n}(p_{n}-y_{n})$
		    \State\label{step:dwifob-least-squares}find $\alpha^{(n)} = (\alpha_0^{(n)},\ldots,\alpha_{m_n}^{(n)})$ that solves
		    \begin{align*}
            &\underset{~~\alpha^{(n)}\in\reals^{m_n+1}}{\mathrm{minimize}} ~~~ \norm{\mathscr{R}_n\alpha^{(n)}}_2^2 + \xi\norm{\mathscr{R}_n^T\mathscr{R}_n}_F\norm{\alpha^{(n)}}_2^2 \\
		            &\mathrm{~~subject ~ to ~~~~}  \mathbf{1}^T\alpha^{(n)}=1
            \end{align*}
            \hspace{5.2mm}where $\mathscr{R}_n = (r_{n-m_n},\ldots,r_n)$ and $r_j = x_{j+1}-y_j$
		    \State $\widehat{u}_{n+1}=x_{n+1}-\sum_{i=0}^{m_n}\alpha_{i}^{(n)}x_{n-m_n+i+1}$
		    \State $\ell_{n}^2=\tfrac{\lambda_n(4-2\lambda_n-\gamma_n\beta)(4-2\lambda_{n+1}-\gamma_{n+1}\beta)}{4\lambda_{n+1}}\norm{p_n-x_n+\tfrac{2\lambda_n+\gamma_n\beta-2}{4-2\lambda_n-\gamma_n\beta}u_n}_M^2$
		    \State $u_{n+1}=\zeta_n|\ell_n|\tfrac{\widehat{u}_{n+1}}{\varepsilon+\norm{\widehat{u}_{n+1}}_M}$ \label{alg-step:scaling}
		    \State $y_{n+1} = x_{n+1}+u_{n+1}$
		\EndFor
	\end{algorithmic}
\label{alg:DWIFOB}
\end{algorithm}

The convergence of \dwifob{} follows from \cref{thm:main}, that shows the convergence of \cref{alg:PFOB}, of which \dwifob{} is a special instance with a specific class of deviations.

\begin{corollary}
Consider the monotone inclusion problem \eqref{eq:monotone_inclusion} and suppose that \cref{assum:monotone_inclusion,assum:parameters} hold. Let  $(x_n)_{n\in\nat}$ be the sequence generated by \cref{alg:DWIFOB}. Then, the sequence $(x_n)_{n\in\nat}$ converges weakly to a point in the solution set $\zer{A+C}$.
\end{corollary}

\section{Primal--dual variant of \dwifob{}}
\label{sec:pd-DWIFOB}

In this section, we consider a specific type of monotone inclusion problems that, after being translated to a primal--dual framework, can be efficiently tackled by \dwifob{}. We propose a primal--dual algorithm based on \cref{alg:DWIFOB} for solving such problems.

\paragraph{Problem statement.} We consider primal inclusion problems of finding $x\in\PrimS$ such that
\begin{align} \label{eq:inclusion_CondatVu}
    0 \in Ax + L^*B(Lx) + Cx
\end{align}
with the following assumptions.

\begin{assumption}
\label{assum:primal_dual_CondatVu}
Assume that
\renewcommand{\labelenumi}{\emph{(\roman{enumi})}}
\begin{enumerate}
    \item $A:\PrimS \rightarrow 2^{\PrimS}$ is a maximally monotone operator;
    \item $B:\mathcal{K} \rightarrow 2^{\mathcal{K}}$ is a maximally monotone operator;
    \item $L:\PrimS \rightarrow \mathcal{K}$ is a bounded linear operator;
    \item $C:\PrimS \rightarrow \PrimS$ is a $\tfrac{1}{\beta}$-cocoercive operator with respect to the metric $\|\cdot \|$;
    \item The solution set $\zer{A+L^*BL+C}$ is nonempty.
\end{enumerate}
\end{assumption}

\paragraph{Translation to a primal--dual framework.} The inclusion problem \eqref{eq:inclusion_CondatVu} can be translated to a primal--dual setting \cite{he2012convergence} to get the inclusion problem
\begin{align} \label{eq:primal_dual_inclusion}
    0\in\mathcal{A}z + \mathcal{C}z
\end{align}
in which, with some abuse of notation,
\begin{align}\label{eq:A_C_caligraphic}
   \mathcal{A} = \begin{bmatrix}A & L^*\\-L & B^{-1}\end{bmatrix} && \mathcal{C} = \begin{bmatrix}C & 0\\0 & 0\end{bmatrix} 
\end{align}
and $z \defeq (x,\mu) \in \PrimS \times \mathcal{K}$ is a primal--dual pair. It holds that $x$ is a solution to \eqref{eq:inclusion_CondatVu} if and only if there exists some $\mu \in \mathcal{K}$ such that $z = (x, \mu)$ is a solution to \eqref{eq:primal_dual_inclusion}.

In this setting, the operator $\mathcal{A}$ is a maximally monotone  \cite[Proposition 26.32]{bauschke2017convex} and  the operator $\mathcal{C}$ is $1/\beta$-cocoercive with respect to the norm $\norm{\cdot}_{{M}}$, with 
\begin{align}\label{eq:VuCondat_M}
    M = \begin{bmatrix}
        I & -\tau L^*\\
        -\tau L & \tau\sigma^{-1}I
    \end{bmatrix},
\end{align}
where $\tau>0$ and $\sigma>0$ are chosen such that $\sigma\tau\|L\|^2 < 1$, which ensures that $M$ is strictly positive. Therefore, the inclusion problem \eqref{eq:primal_dual_inclusion} can be solved using the \dwifob{} algorithm. \Cref{alg:PD-DWIFOB} describes our primal--dual \dwifob{} algorithm which is derived by a straightforward application of \dwifob{} to \eqref{eq:primal_dual_inclusion}. With  $C=0$ and $m=1$, this algorithm is equivalent to \cite[Algorithm~4]{nofob-increments}, an inertial primal--dual algorithm.

\begin{algorithm}[h!]
	\caption{}
	\begin{algorithmic}[1]
	    \State \textbf{Input:} $(x_0,\mu_0)\in\PrimS\times\mathcal{K}$; $m\geq1$;  the sequences \seq{\lambda}{n} and \seq{\zeta}{n} as defined in \cref{assum:parameters}; the regularization parameter $\xi$; $\sigma>0, \tau>0$ such that $\sigma\tau\|L\|^2 < 1$; and $\varepsilon\geq0$.
	    \State \textbf{set} $(\widehat{x}_0,\widehat{\mu}_0) = (x_0,\mu_0)$  and $(u_{x,0},u_{\mu,0})=(0,0)$
		\For {$n=0,1,2,\ldots$}
		    \State $m_n = \min(m,n)$
		    \State $p_{x,n}=J_{\tau A}(\widehat{x}_n-\tau L^*\widehat{\mu}_n-\tau C\widehat{x}_n)$ \label{alg-line:alg4-res1}
		    \State $p_{\mu,n}=J_{\sigma B^{-1}}\lp\widehat{\mu}_n+\sigma{L}(2p_{x,n}-\widehat{x}_n)\rp$ \label{alg-line:alg4-res2}
		    \State $x_{n+1}=x_{n}+\lambda_{n}(p_{x,n}-\widehat{x}_{n})$ \label{alg-line:alg4-relax1}
		    \State $\mu_{n+1}=\mu_{n}+\lambda_{n}(p_{\mu,n}-\widehat{\mu}_{n})$ \label{alg-line:alg4-relax2}
		    \State find $\alpha^{(n)} = (\alpha_0^{(n)},\ldots,\alpha_{m_n}^{(n)})$ that solves
		    \begin{equation*}
                \begin{aligned}
            &\underset{~~\alpha^{(n)}\in\reals^{m_n+1}}{\mathrm{minimize}} ~~~ \norm{\mathscr{R}_n\alpha^{(n)}}_2^2 + \xi\norm{\mathscr{R}_n^T\mathscr{R}_n}_F\norm{\alpha^{(n)}}_2^2 \\
		            &\mathrm{~~subject ~ to ~~~~}  \mathbf{1}^T\alpha^{(n)}=1
            \end{aligned}
            \end{equation*}
            \hspace{5.2mm}where $\mathscr{R}_n = (r_{n-m_n},\ldots,r_n)$ where $r_j = (x_{j+1}-\widehat{x}_j,\mu_{j+1}-\widehat{\mu}_j)$ 
		    \State $\begin{bmatrix}\widehat{u}_{x,n+1}\\\widehat{u}_{\mu,n+1}\end{bmatrix}=\begin{bmatrix}x_{n+1}\\\mu_{n+1}\end{bmatrix}-\sum_{i=0}^{m_n}{\alpha_{i}^{(n)}\begin{bmatrix}x_{n-m_n+i+1}\\\mu_{n-m_n+i+1}\end{bmatrix}}$ \label{alg-line:alg4-uhat}
		    \State $\ell_{n}^2=\tfrac{\lambda_n(4-2\lambda_n-\tau\beta)(4-2\lambda_{n+1}-\tau\beta)}{4\lambda_{n+1}}\norm{\begin{bmatrix}p_{x,n}\\p_{\mu,n}\end{bmatrix}-\begin{bmatrix}x_{n}\\\mu_{n}\end{bmatrix}+\tfrac{2\lambda_n+\tau\beta-2}{4-2\lambda_n-\tau\beta}\begin{bmatrix}u_{x,n}\\u_{\mu,n}\end{bmatrix}}_M^2$ \label{alg-line:alg4-ell}
		    \State $\begin{bmatrix}u_{x,n+1}\\u_{\mu,n+1}\end{bmatrix}=\tfrac{\zeta_n\vert\ell_n\vert}{\varepsilon+\norm{(\widehat{u}_{x,n+1},\widehat{u}_{{\mu},n+1})}_M}\begin{bmatrix}\widehat{u}_{x,n+1}\\\widehat{u}_{\mu,n+1}\end{bmatrix}$ \label{alg-line:alg4-u}
		    \State $\widehat{x}_{n+1} = x_{n+1}+u_{x,n+1}$ \label{alg-line:alg4-extraploation}
		    \State $\widehat{\mu}_{n+1} = \mu_{n+1}+u_{\mu,n+1}$
		\EndFor
	\end{algorithmic}
\label{alg:PD-DWIFOB}
\end{algorithm}

The following is a result on weak convergence of the iterates generated by \cref{alg:PD-DWIFOB}. It is based on showing that \cref{alg:PD-DWIFOB} is a special case of the weakly convergent \cref{alg:PFOB}. 

\begin{corollary}\label{cor:pd_dwifob_var1}
Consider the monotone inclusion problem \eqref{eq:inclusion_CondatVu} under \cref{assum:primal_dual_CondatVu} and suppose that \cref{assum:parameters} holds. Then the sequence \seq{x}{n} in \cref{alg:PD-DWIFOB} converges weakly to a point in $\zer{A+L^*BL+C}$.
\end{corollary}

\begin{proof}
Comparing \cref{alg:PD-DWIFOB}  with \cref{alg:PFOB}, we set $p_n=(p_{x,n},p_{\mu,n})$, $y_n=(\widehat{x}_n,\widehat{\mu}_n)$, define $\mathcal{A}$ and $\mathcal{C}$ as in \eqref{eq:A_C_caligraphic}, and let $M$ be defined as in \eqref{eq:VuCondat_M}. Then, we have the following update
\begin{align*}
    p_{n}=(p_{x,n},p_{\mu,n})&=(M + \tau\mathcal{A})^{-1}(M y_n - \tau\mathcal{C}y_n)\\
    &=\begin{bmatrix}I+\tau A&0\\-2\tau L&\tau\sigma^{-1}I+\tau B^{-1}\end{bmatrix}^{-1} \begin{bmatrix}\widehat{x}_n -\tau L^*\widehat{\mu}_n - \tau C\widehat{x}_n\\-\tau L\widehat{x}_n + \tau\sigma^{-1}\widehat{\mu}_n\end{bmatrix}\\
    &=\begin{bmatrix}
        (I+\tau A)^{-1}(\widehat{x}_n-\tau L^*\widehat{\mu}_n - \tau C\widehat{x}_n)\\
        (I+\sigma B^{-1})^{-1}(\widehat{\mu}_n+\sigma L(2p_{x,n}-\widehat{x}_n))
    \end{bmatrix}\\
    &= \begin{bmatrix} J_{\tau A} \left(\widehat{x}_n-\tau L^*\widehat{\mu}_n - \tau C\widehat{x}_n \right)\\J_{\sigma B^{-1}} \left(\widehat{\mu}_n+\sigma L(2p_{x,n}-\widehat{x}_n)\right) \end{bmatrix},
\end{align*}
which gives the resolvent steps of \cref{alg:PD-DWIFOB} (\cref{alg-line:alg4-res1,alg-line:alg4-res2}). Moreover, it is also straightforward to verify that,  by substituting $(x_{n+1},\mu_{n+1})$ in place of $x_{n+1}$ in \cref{alg:PFOB}, the relaxation steps of \cref{alg:PD-DWIFOB} (\cref{alg-line:alg4-relax1,alg-line:alg4-relax2})  are equivalent to that of \cref{alg:PFOB}. Additionally, with the devised choice of $u_{n+1}=(u_{x,n+1},u_{\mu,n+1})$ in \cref{alg:PD-DWIFOB}, the norm condition of \cref{alg:PFOB} holds. Therefore, since \cref{alg:PD-DWIFOB} is a special instance of \cref{alg:PFOB} and due to equivalence of \eqref{eq:inclusion_CondatVu} and \eqref{eq:primal_dual_inclusion}, a direct application of \cref{thm:main} concludes the proof.
\end{proof}

\begin{remark}\label{rem:CP}
For the choice of $\lambda_n=1$, $u_{x,n} = 0$ and $u_{\mu,n}=0$ for all $n\in\nat$ and $C = 0$,  \cref{alg:PD-DWIFOB} reduces to the standard Chambolle--Pock iteration \cite{chambolle2011first}, that is
\begin{align*}
    (x_{n+1},\mu_{n+1}) = \begin{bmatrix} J_{\tau A} \left({x_n}-\tau L^*{\mu_n} \right)\\
    J_{\sigma B^{-1}} \left({\mu}_n+\sigma L(2{x_{n+1}}-{x}_n)\right) \end{bmatrix}.
\end{align*}
\end{remark}

\subsection{Efficient evaluation of the $M$-induced norm}\label{subsec:efficient_evaluation}
In \cref{alg:PD-DWIFOB}, we need to evaluate two $M$-induced norms per iteration, where $M$ is given by \eqref{eq:VuCondat_M}. This means that, in addition to evaluating $L$ and $L^*$ in the resolvent steps, two extra evaluations each of $L$ and $L^*$ are needed due the $M$-induced norms. These extra evaluations can be computationally expensive, which would make the algorithm computationally inefficient. However, by utilizing a similar approach as in \cite[Section 6.1]{nofob-increments}, the extra evaluations can be efficiently done by reusing some of the previous computations.

We next show that we only need to apply $L$ and $L^*$ once per iteration (except for the first) in \cref{alg:PD-DWIFOB}. Observe that, by applying the operator $L$ on \cref{alg-line:alg4-relax1,alg-line:alg4-uhat,alg-line:alg4-extraploation} (after substitution of \cref{alg-line:alg4-u}) of  \cref{alg:PD-DWIFOB}, we obtain the following relations
\begin{equation}\label{eq:recursion}
    \begin{aligned}
        Lx_{n+1} &= Lx_{n} + \lambda_{n}(Lp_{x,n}-L\widehat{x}_n),\\
        L\widehat{u}_{x,n+1}&=Lx_{n+1}-\sum_{i=0}^{m_n}\alpha_i^{(n)}Lx_{n-m_n+i+1},\\
        L\widehat{x}_{n+1}&=Lx_{n+1}+\tfrac{\zeta_n|\ell_n|}{\varepsilon+\norm{(\widehat{u}_{x,n+1},\widehat{u}_{\mu,n+1})}_M}L\widehat{u}_{x,n+1}.
    \end{aligned}
\end{equation}
In these relations, for all $n>0$, we only need to evaluate $Lp_{x,n}$. The rest of the quantities to the right-hand sides of the above relations are already computed and can be reused. This means that, in practice, we only need to only evaluate one of each $L$ (for $Lp_{x,n}$) and $L^*$ (for  $L^*\widehat{\mu}_{n}$) at each iteration, except for the first. Therefore, since the most computationally expensive part of our algorithm often is evaluating $L$ and $L^*$, exploiting this technique keeps the computational cost of our algorithm similar to that of the Chambolle--Pock method. However, in order to use this approach, one needs to store $m_n+4$ vectors of the same dimension as the dual variable. Hence, in applications where storage is a bottleneck, using a large $m_n$ might be restrictive.

Evaluation of the $M$-induced norm of, for instance, $\norm{(\widehat{u}_{x,n},\widehat{u}_{\mu,n})}_M$ can be done  as
\begin{align}\label{eq:recursive_norm}
    \norm{(\widehat{u}_{x,n},\widehat{u}_{\mu,n})}_M^2 &= \norm{\widehat{u}_{x,n}}^2 + \tfrac{\tau}{\sigma}\norm{\widehat{u}_{\mu,n}}^2 - 2\tau\inpr{\widehat{u}_{\mu,n}}{L\widehat{u}_{x,n}},
\end{align}
where $L\widehat{u}_{x,n}$ is already available from the stored set of quantities. The other $M$-induced norm in \cref{alg-line:alg4-ell} of \cref{alg:PD-DWIFOB} can be computed in the same way as above without extra evaluations of $L$ or $L^*$.

%%%%%%%%%%%%%%%%%%%%%%%%%%%%%%%%%%%%%%%
\section{Numerical experiments}
\label{sec:num-exp}

In this section, we evaluate the performance of the primal--dual variant of the \dwifob{} algorithm and compare it with the Chambolle--Pock primal--dual method and RAA.

We consider a \emph{support vector machine} (SVM) problem with $l_1$-norm regularization for classification of the form
\begin{equation}\label{eq:l1_reg_svm}
    \underset{(w,b)\in\reals^d\times\reals}{\text{minimize}}~~\sum_{i=1}^N \max\left(0,1-\phi_i(w^T \theta_i + b)\right) + \delta\|w\|_1
\end{equation}
given a labeled training data set $\{(\theta_i,\phi_i) \}_{i=1}^N$, where $\theta_i\in\reals^d$ and $\phi_i\in\{-1,1\}$ are training data and labels respectively, $\delta>0$ is the regularization parameter, and $x=(w,b)$ with $b\in\reals$ and $w\in\reals^d$ is the decision variable. This problem can be reformulated as
\begin{align}\label{eq:l1_reg_svm_reform}
    \underset{ x\in \reals^{d+1} }{\text{minimize }} f(Lx)+g(x)
\end{align}
with
\begin{align*}
    f(y) = \sum_{i=1}^N \max\left(0,1-y_i\right),
    %\mathbf{1}^T \max(\mathbf{0}, \mathbf{1}-Lx),
    && g(x) = \delta\|\omega\|_1, && L = \begin{bmatrix} \phi_1\theta_1^T & \phi_1\\ \vdots & \vdots \\ \phi_N\theta_N^T & \phi_N\end{bmatrix},
\end{align*}
where $f$, $g \colon \reals^{d+1} \to \reals$ are proper, closed, and convex (and non-smooth) functions with full domain and $L$ is a bounded linear operator. A point $x^\star\in\reals^{d+1}$ solves problem~\eqref{eq:l1_reg_svm_reform} if and only if
\begin{equation}\label{eq:pd-optimality}
    0\in L^*\partial f(Lx) + \partial g(x),
\end{equation}
where $\partial f$ and $\partial g$ are the subdifferentials of $f$ and $g$, respectively \cite[Proposition~16.42]{bauschke2017convex}. By  \cite[Theorem~20.25]{bauschke2017convex},  $\partial f$ and $\partial g$ are maximally monotone. Therefore, we solve the monotone inclusion \eqref{eq:pd-optimality} in order to find a solution to problem~\eqref{eq:l1_reg_svm_reform}, which, by setting $A=\partial{g}$, $B=\partial{f}$, and $C=0$, fits into the framework of problem \eqref{eq:primal_dual_inclusion}. We use the following algorithms to solve the problem:
\begin{itemize}
    \item Chambolle and Pock's primal--dual method  (CP) \cite{chambolle2011first};
    \item The  primal--dual \dwifob{} method in \cref{alg:PD-DWIFOB} (Alg\labelcref{alg:PD-DWIFOB});
    %\item the second variant of primal--dual \dwifob{}, \cref{alg:P3DWIFOB} (Alg\labelcref{alg:P3DWIFOB});
    \item Regularized Anderson acceleration (RAA), \cref{alg:AA}, \cite{scieur2020regularized,walker2011anderson}, applied to the fixed-point map of Chambolle--Pock, see \cref{rem:CP}.
\end{itemize}

In the algorithms listed above, evaluating $L$ and $L^*$ in the resolvent steps and solving the least-squares problem, if there is one,  are the computationally intensive parts. Since the Chambolle--Pock algorithm does not involve solving a least-squares problem, it has a cheaper per-iteration cost compared to the other algorithms. To provide a fair comparison, we compare the methods using \emph{scaled iterations}. Let $\ccp$ and $\calg$ be the average per-iteration computational cost of the Chambolle--Pock method and one of the algorithms mentioned above ($\mathrm{alg} \in \set{\mathrm{CP}, \text{Alg\labelcref{alg:PD-DWIFOB}}, \text{RAA}}$), respectively. The scaled iteration is the iteration count scaled by the ratio $\tfrac{\calg}{\ccp}$. The iteration costs $\ccp$ and $\calg$ are numerically approximated by measuring the average per-iteration elapsed time of the individual algorithms. The benefits of using the notion of scaled iteration are two-fold. In addition to considering the relative per-iteration computational cost of the algorithms, it eliminates the impact of computational capacity/power of the platform that the algorithms are implemented on, which makes the results more reproducible.

The experiments are done using three different benchmark datasets; the \emph{breast cancer} dataset  with 683 samples and 10 features, the \emph{sonar} dataset with 208 samples and 60 features, and \emph{colon cancer} dataset with 62 samples and 2000 features, all from \cite{chang2011libsvm}.  The numerical experiments are done on a laptop with a $1.4$ GHz Quad-core Intel Core i5 processor with $16$ GB of memory. The algorithms are implemented using the Julia programming language (Version 1.3.1).

In all experiments, the primal and the dual step-size parameters are chosen as $\tau=\sigma=0.99/\norm{L}^2$, $\zeta_n=0.99$ for all $n\in\nat$, $\varepsilon=0$, and a fixed relaxation parameter $\lambda = 1.0$ for \cref{alg:PD-DWIFOB} is used. Unless otherwise stated, the algorithms are initialized at $(x_0, \mu_0) = 0$. We report results from the numerical experiments in a sequence of figures. The $M$-induced distance to a solution is used as the convergence measure where the individual underlying solutions are found by running the standard Chambolle--Pock algorithm until  $\norm{x_n-x_{n-1}}\leq10^{-15}
$ and $\norm{\mu_n-\mu_{n-1}}\leq10^{-15}$. All algorithms that converge do so to the same solution. Moreover, all evaluations of $L$, $L^*$, and $\norm{\cdot}_M$ are done using the proposed recursive method of \cref{subsec:efficient_evaluation}, unless otherwise stated.

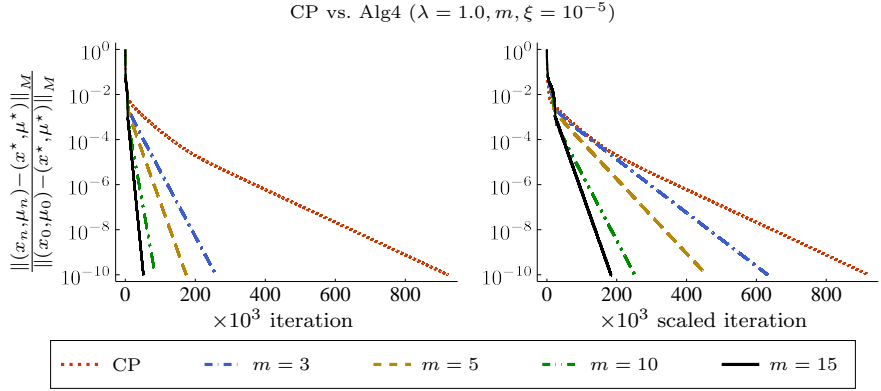
\begin{figure}
    \centering
    \begin{subfigure}{\textwidth}
    \centering
    \begin{tikzpicture}[scale=1]
        \node[text=black](title) at (0.05\linewidth,0.33\linewidth) {{\scriptsize CP vs.\ Alg\labelcref{alg:PD-DWIFOB} ($\lambda=1.0,m,\xi=10^{-5}$)}};
    \end{tikzpicture}
    \end{subfigure}
    \begin{subfigure}{0.49\textwidth}
    \centering
    \begin{tikzpicture}[scale=1.0]
      \node[inner sep=0pt] (fig) at (0,0) {\includegraphics[ width = 0.92\linewidth,keepaspectratio]{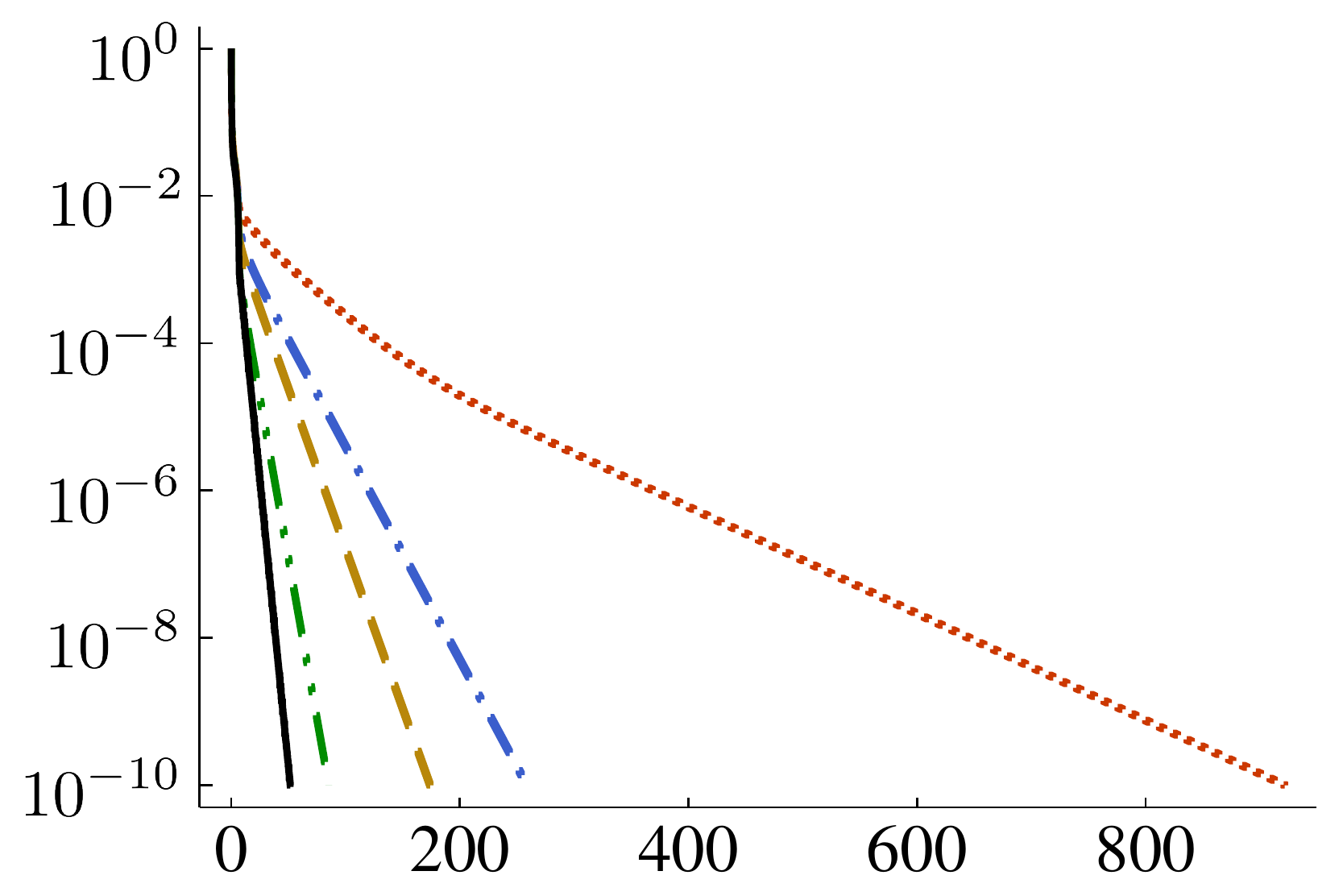}};
        \node[text=black](xlabel) at (0.05\linewidth,-.33\linewidth) {{\footnotesize $\times 10^3$ iteration}};
        \node[rotate=90, text=black] (ylabel) at (-.51\linewidth,0) {$\frac{\norm{(x_n,\mu_n)-(x^\star,\mu^\star)}_M}{\norm{(x_0,\mu_0)-(x^\star,\mu^\star)}_M}$};
    \end{tikzpicture}
    \end{subfigure}
    \begin{subfigure}{0.49\textwidth}
    \centering
    \begin{tikzpicture}[scale=1.0]
        \node[inner sep=0pt] (fig) at (0,0) {\includegraphics[ width = 0.92\linewidth,keepaspectratio]{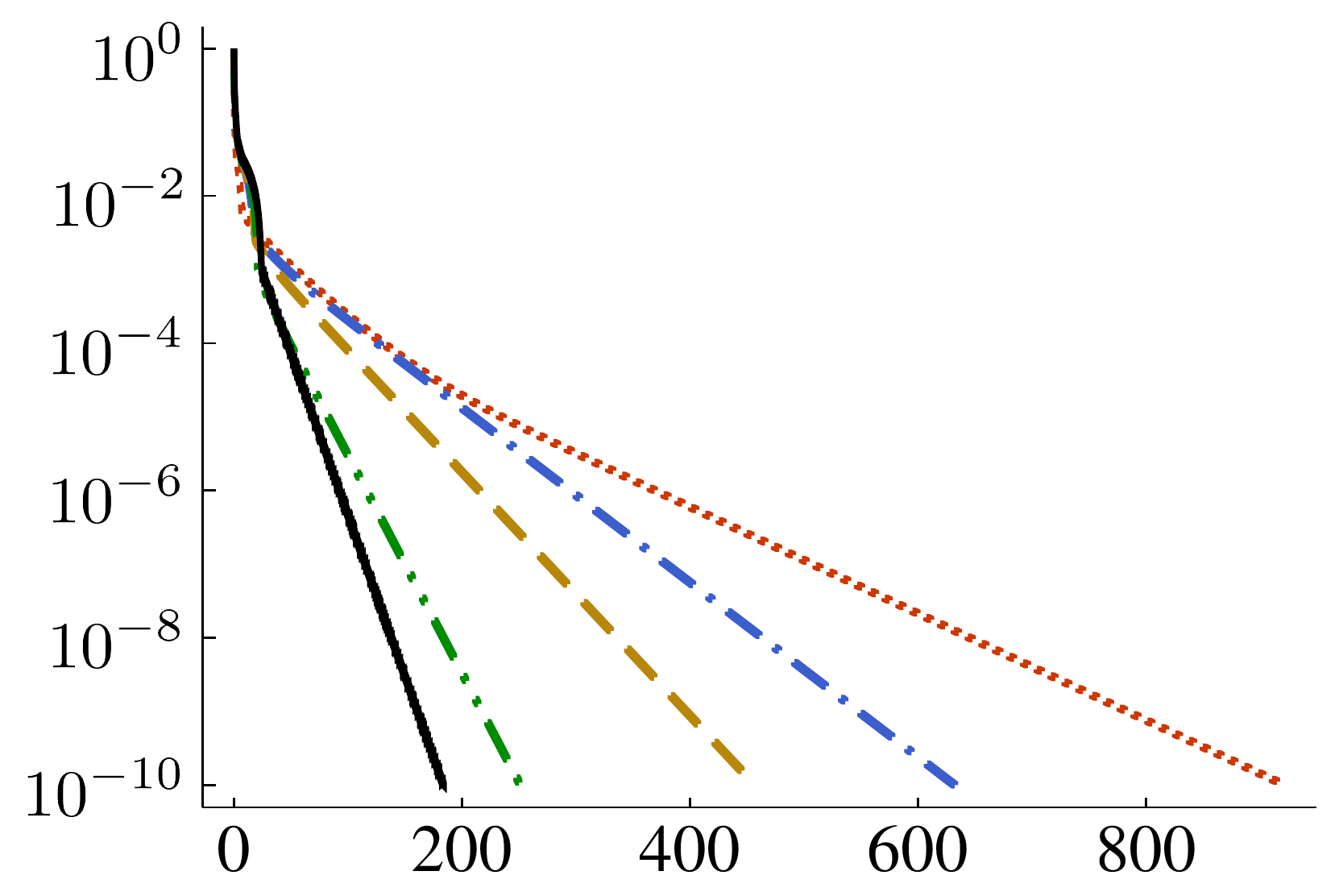}};
        \node[text=black](xlabel) at (0.05\linewidth,-.33\linewidth) {{\footnotesize $\times 10^3$ scaled iteration}};
    \end{tikzpicture}
    \end{subfigure}
    \begin{subfigure}{\textwidth}
    \centering
    \begin{tikzpicture}[scale=1]
        \matrix[draw=black, very thin]  at (-4,0) { %\matrix [draw=black, very thin]
            \draw[redd, dotted, very thick](0,0)--(0.54,0);&\node {{\scriptsize CP}};&
            \node {~~~~~};&
            \draw[bluee, dashdotted, very thick](0,0)--(0.5,0);&\node {{\scriptsize $m=3$}};&
            \node {~~~~~};&
            \draw[brownn, dashed, very thick](0,0)--(0.5,0);&\node {{\scriptsize $m=5$}};&
            \node {~~~~~};&
            \draw[greenn, dashdotdotted, very thick](0,0)--(0.5,0);&\node {{\scriptsize $m=10$}};&
            \node {~~~~~};&
            \draw[blackk, solid, very thick](0,0)--(0.5,0);&\node {{\scriptsize $m=15$}};\\
        };
    \end{tikzpicture}
    \end{subfigure}
\caption{Normalized $M$-induced distance to the solution vs.\ iteration number (\emph{left}) and scaled iteration number (\emph{right})  for the $l_1$-norm regularized SVM, problem \eqref{eq:l1_reg_svm}, with $\delta = 0.5$, on the \emph{breast cancer} dataset \cite{chang2011libsvm} with 683 samples and 10 features. Solved using the Chambolle--Pock algorithm and Alg\labelcref{alg:PD-DWIFOB} ($\lambda=1.0$, $m$, $\xi=10^{-5}$) for several memory sizes $m$, all with $\tau=\sigma = 0.99/\norm{L}$. }
\label{fig:CP_Alg4_m_breastCancer}
\end{figure}

\begin{figure}
    \centering
    \begin{subfigure}{\textwidth}
    \centering
    \begin{tikzpicture}[scale=1]
        \node[text=black](title) at (0.05\linewidth,0.33\linewidth) {{\scriptsize CP vs.\ Alg\labelcref{alg:PD-DWIFOB} ($\lambda=1.0,m,\xi=10^{-5}$)}};
    \end{tikzpicture}
    \end{subfigure}
    \begin{subfigure}{0.49\textwidth}
    \centering
    \begin{tikzpicture}[scale=1.0]
      \node[inner sep=0pt] (fig) at (0,0) {\includegraphics[ width = 0.92\linewidth,keepaspectratio]{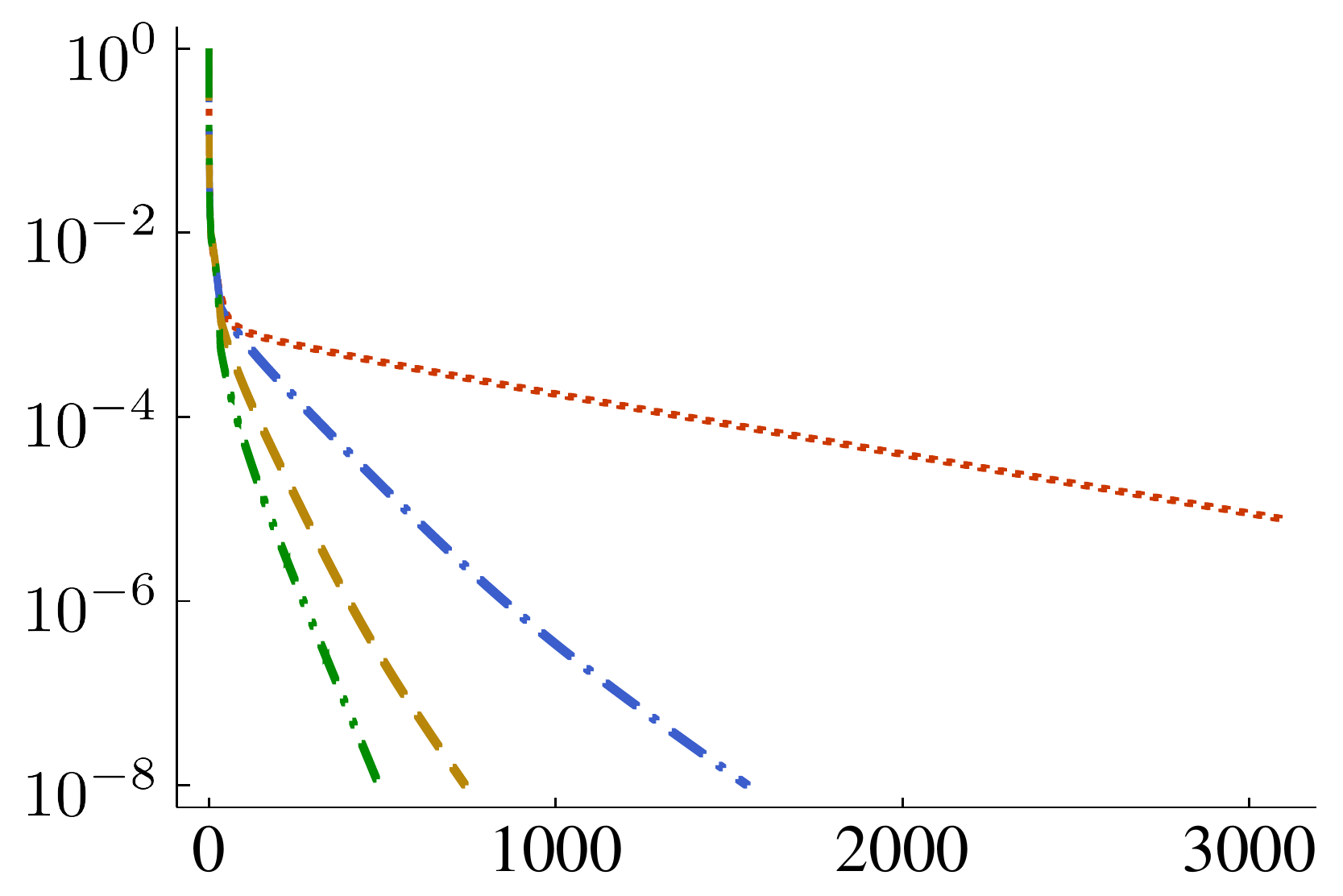}};
        \node[text=black](xlabel) at (0.05\linewidth,-.33\linewidth) {{\footnotesize $\times 10^3$ iteration}};
        \node[rotate=90, text=black] (ylabel) at (-.51\linewidth,0) {$\frac{\norm{(x_n,\mu_n)-(x^\star,\mu^\star)}_M}{\norm{(x_0,\mu_0)-(x^\star,\mu^\star)}_M}$};
    \end{tikzpicture}
    \end{subfigure}
    \begin{subfigure}{0.49\textwidth}
    \centering
    \begin{tikzpicture}[scale=1.0]
        \node[inner sep=0pt] (fig) at (0,0) {\includegraphics[ width = 0.92\linewidth,keepaspectratio]{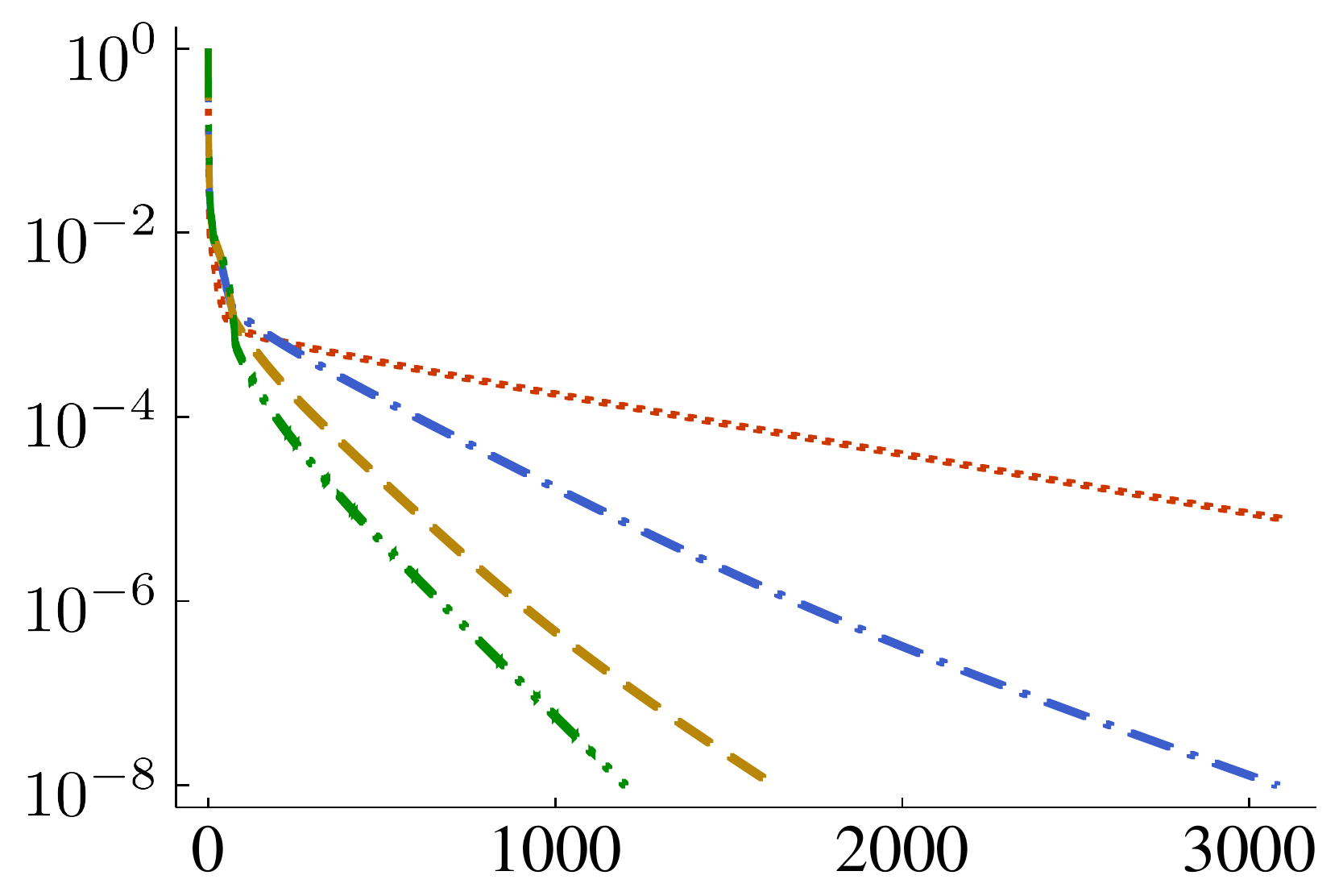}};
        \node[text=black](xlabel) at (0.05\linewidth,-.33\linewidth) {{\footnotesize $\times 10^3$ scaled iteration}};
    \end{tikzpicture}
    \end{subfigure}
    \begin{subfigure}{\textwidth}
    \centering
    \begin{tikzpicture}[scale=1]
        \matrix[draw=black, very thin]  at (-4,0) { %\matrix [draw=black, very thin]
            \draw[redd, dotted, very thick](0,0)--(0.54,0);&\node {{\scriptsize CP}};&
            \node {~~~~~};&
            \draw[bluee, dashdotted, very thick](0,0)--(0.5,0);&\node {{\scriptsize $m=5$}};&
            \node {~~~~~};&
            \draw[brownn, dashed, very thick](0,0)--(0.5,0);&\node {{\scriptsize $m=10$}};&
            \node {~~~~~};&
            \draw[greenn, dashdotdotted, very thick](0,0)--(0.5,0);&\node {{\scriptsize $m=15$}};\\
        };
    \end{tikzpicture}
    \end{subfigure}
\caption{Normalized $M$-induced distance to the solution vs.\ iteration number (\emph{left}) and scaled iteration number (\emph{right})  for the $l_1$-norm regularized SVM, problem \eqref{eq:l1_reg_svm}, with $\delta = 1.0$, on the \emph{sonar} dataset \cite{chang2011libsvm} with 208 samples and 60 features. Solved using the Chambolle--Pock algorithm and Alg\labelcref{alg:PD-DWIFOB} ($\lambda=1.0$, $m$, $\xi=10^{-5}$) for several memory sizes $m$, all with $\tau=\sigma = 0.99/\norm{L}$. }
\label{fig:CP_Alg4_m_sonar}
\end{figure}

\begin{figure}
    \centering
    \begin{subfigure}{\textwidth}
    \centering
    \begin{tikzpicture}[scale=1]
        \node[text=black](title) at (0.05\linewidth,0.33\linewidth) {{\scriptsize CP vs.\ Alg\labelcref{alg:PD-DWIFOB} ($\lambda=1.0,m,\xi=10^{-6}$)}};
    \end{tikzpicture}
    \end{subfigure}
    \begin{subfigure}{0.49\textwidth}
    \centering
    \begin{tikzpicture}[scale=1.0]
      \node[inner sep=0pt] (fig) at (0,0) {\includegraphics[ width = 0.92\linewidth,keepaspectratio]{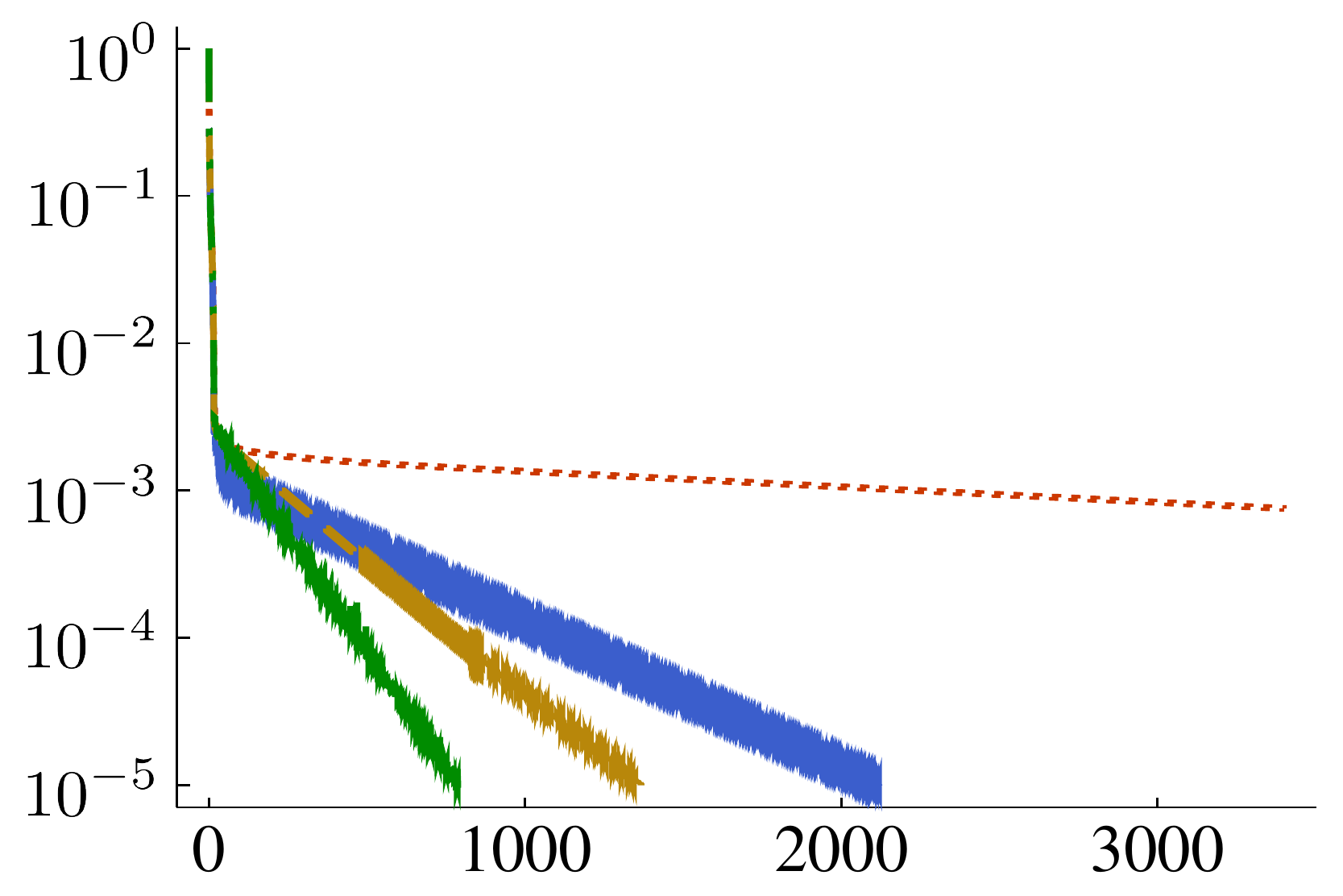}};
        \node[text=black](xlabel) at (0.05\linewidth,-.33\linewidth) {{\footnotesize $\times 10^3$ iteration}};
        \node[rotate=90, text=black] (ylabel) at (-.51\linewidth,0) {$\frac{\norm{(x_n,\mu_n)-(x^\star,\mu^\star)}_M}{\norm{(x_0,\mu_0)-(x^\star,\mu^\star)}_M}$};
    \end{tikzpicture}
    \end{subfigure}
    \begin{subfigure}{0.49\textwidth}
    \centering
    \begin{tikzpicture}[scale=1.0]
        \node[inner sep=0pt] (fig) at (0,0) {\includegraphics[ width = 0.92\linewidth,keepaspectratio]{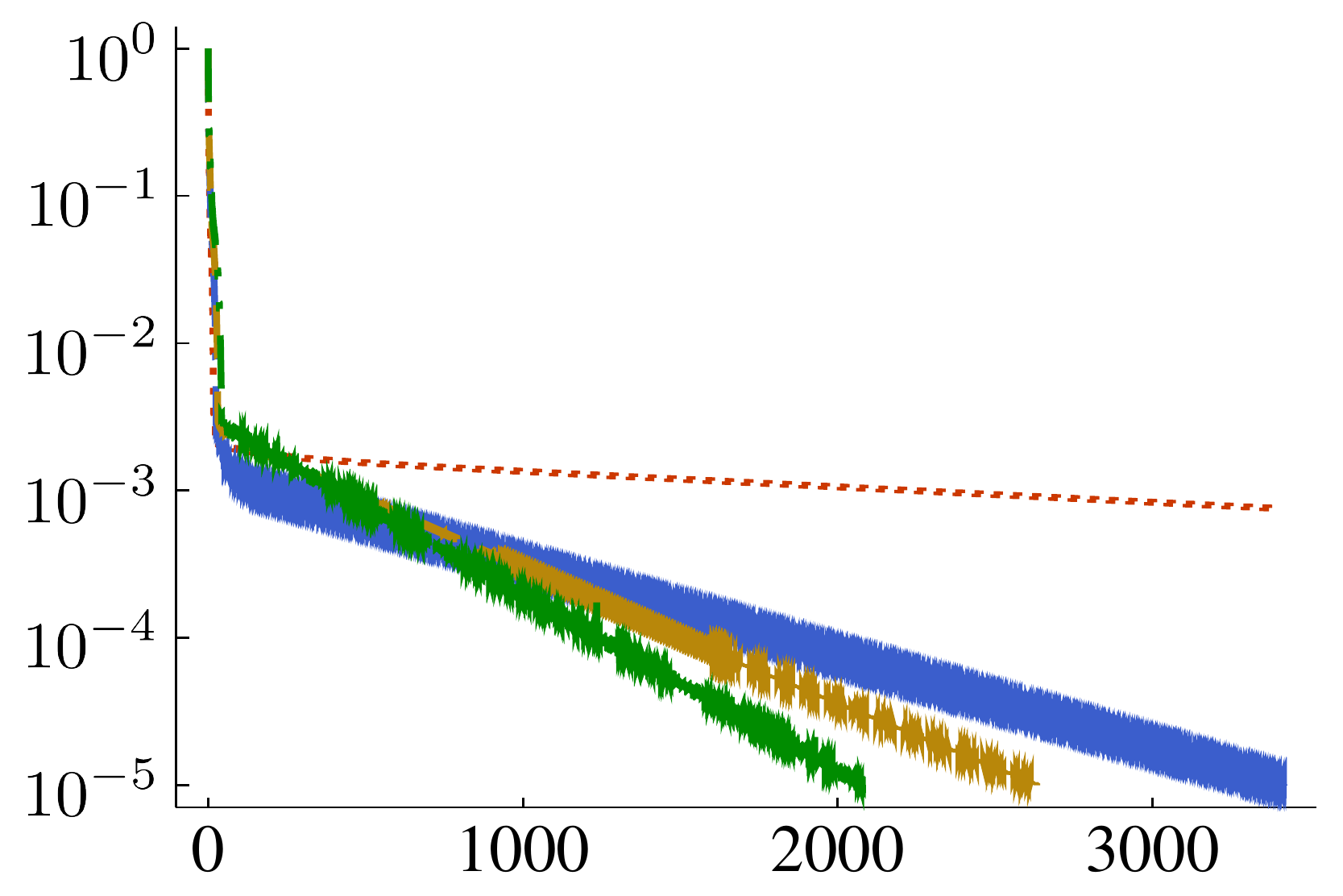}};
        \node[text=black](xlabel) at (0.05\linewidth,-.33\linewidth) {{\footnotesize $\times 10^3$ scaled iteration}};
    \end{tikzpicture}
    \end{subfigure}
    \begin{subfigure}{\textwidth}
    \centering
    \begin{tikzpicture}[scale=1]
        \matrix[draw=black, very thin]  at (-4,0) { %\matrix [draw=black, very thin]
            \draw[redd, dotted, very thick](0,0)--(0.54,0);&\node {{\scriptsize CP}};&
            \node {~~~~~};&
            \draw[bluee, dashdotted, very thick](0,0)--(0.5,0);&\node {{\scriptsize $m=5$}};&
            \node {~~~~~};&
            \draw[brownn, dashed, very thick](0,0)--(0.5,0);&\node {{\scriptsize $m=15$}};&
            \node {~~~~~};&
            \draw[greenn, dashdotdotted, very thick](0,0)--(0.5,0);&\node {{\scriptsize $m=25$}};\\
        };
    \end{tikzpicture}
    \end{subfigure}
\caption{Normalized $M$-induced distance to the solution vs.\ iteration number (\emph{left}) and scaled iteration number (\emph{right})  for the $l_1$-norm regularized SVM, problem \eqref{eq:l1_reg_svm}, with $\delta = 0.1$, on the \emph{colon cancer} dataset \cite{chang2011libsvm} with 62 samples and 2000 features. Solved using the Chambolle--Pock algorithm and Alg\labelcref{alg:PD-DWIFOB} ($\lambda=1.0$, $m$, $\xi=10^{-6}$) for several memory sizes, $m$, all with $\tau=\sigma = 0.99/\norm{L}$. }
\label{fig:CP_Alg4_m_colonCancer}
\end{figure}

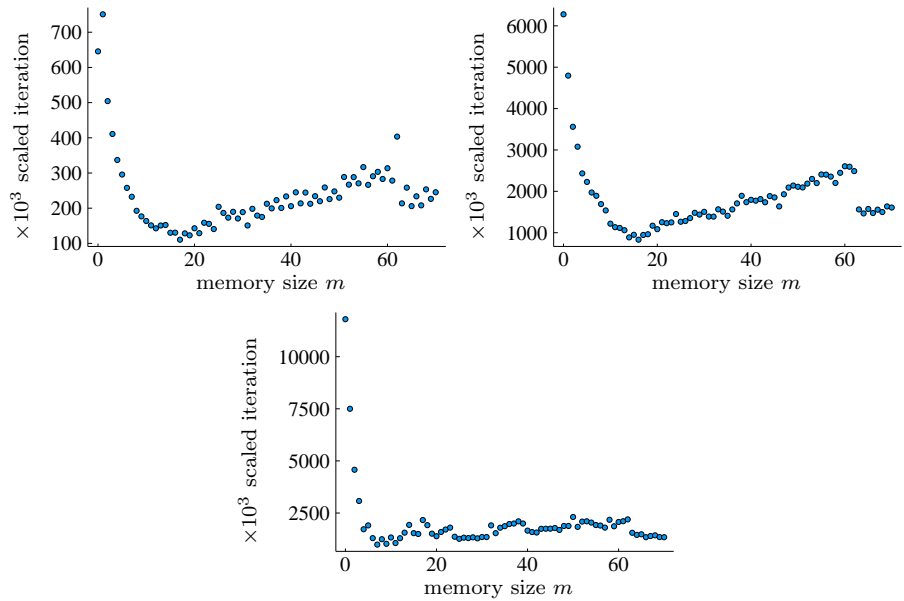
\begin{figure}
    \centering
    \begin{subfigure}{0.49\textwidth}
    \centering
    \begin{tikzpicture}[scale=1.0]
      \node[inner sep=0pt] (fig) at (0,0) {\includegraphics[ width = 0.92\linewidth,keepaspectratio]{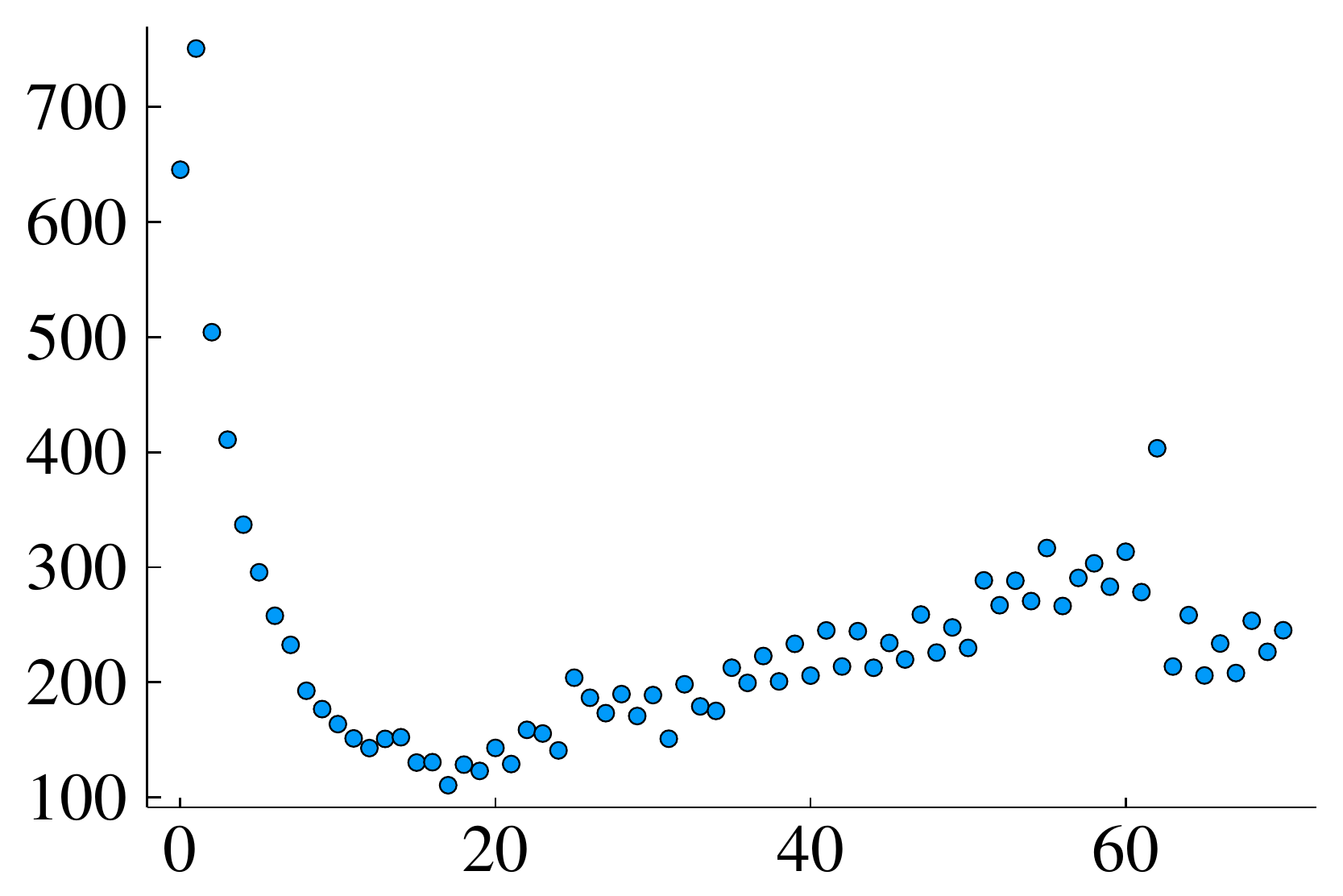}};
        \node[text=black](xlabel) at (0.05\linewidth,-.33\linewidth) {{\footnotesize memory size $m$}};
        \node[rotate=90, text=black] (ylabel) at (-.51\linewidth,0) {{\footnotesize $\times 10^3$ scaled iteration}};
    \end{tikzpicture}
    \end{subfigure}
    \begin{subfigure}{0.49\textwidth}
    \centering
    \begin{tikzpicture}[scale=1.0]
        \node[inner sep=0pt] (fig) at (0,0) {\includegraphics[ width = 0.92\linewidth,keepaspectratio]{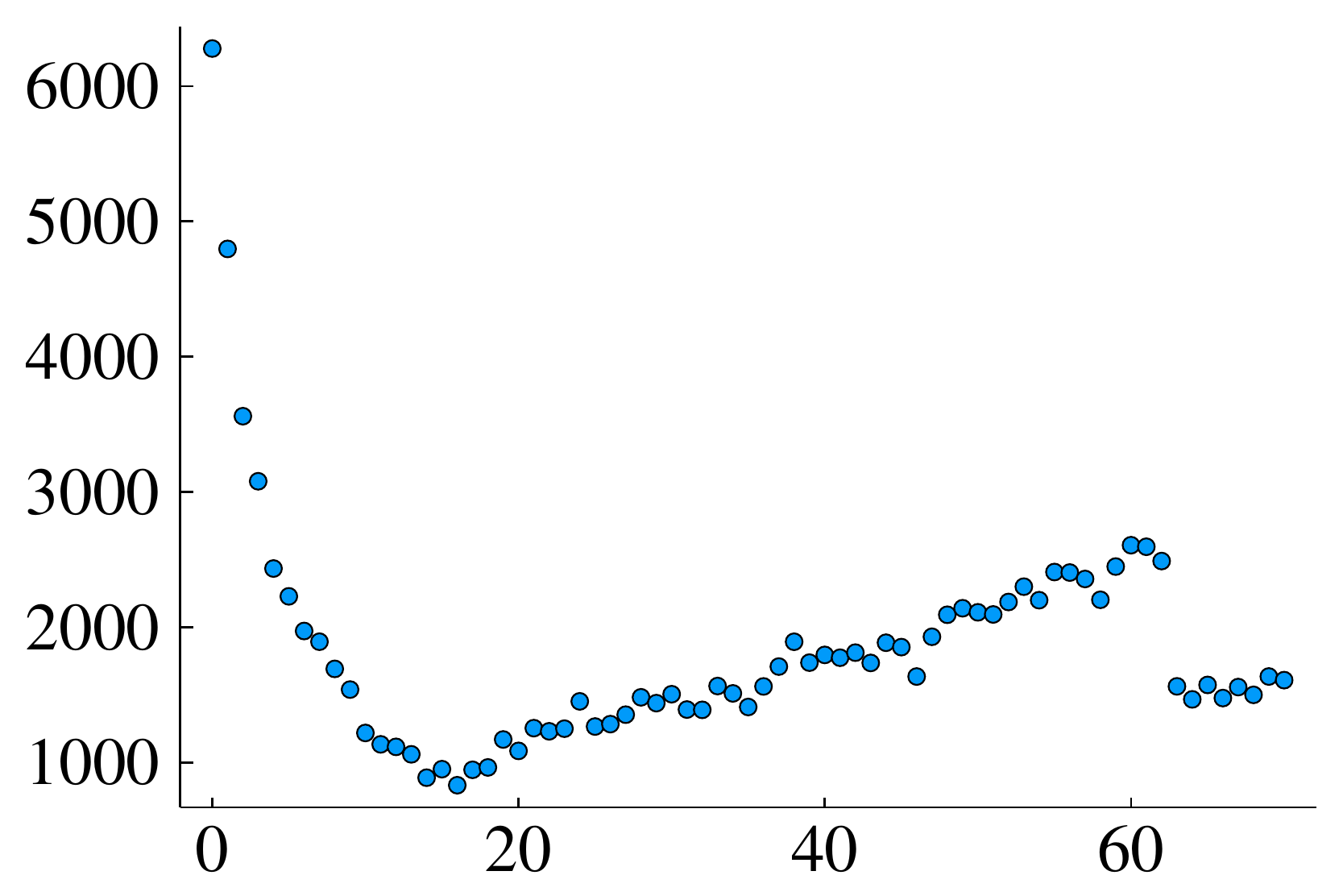}};
        \node[text=black](xlabel) at (0.05\linewidth,-.33\linewidth) {{\footnotesize memory size $m$}};
        \node[rotate=90, text=black] (ylabel) at (-.51\linewidth,0) {{\footnotesize $\times 10^3$ scaled iteration}};
    \end{tikzpicture}
    \end{subfigure}
    \begin{subfigure}{0.49\textwidth}
    \centering
    \begin{tikzpicture}[scale=1.0]
        \node[inner sep=0pt] (fig) at (0,0) {\includegraphics[ width = 0.92\linewidth,keepaspectratio]{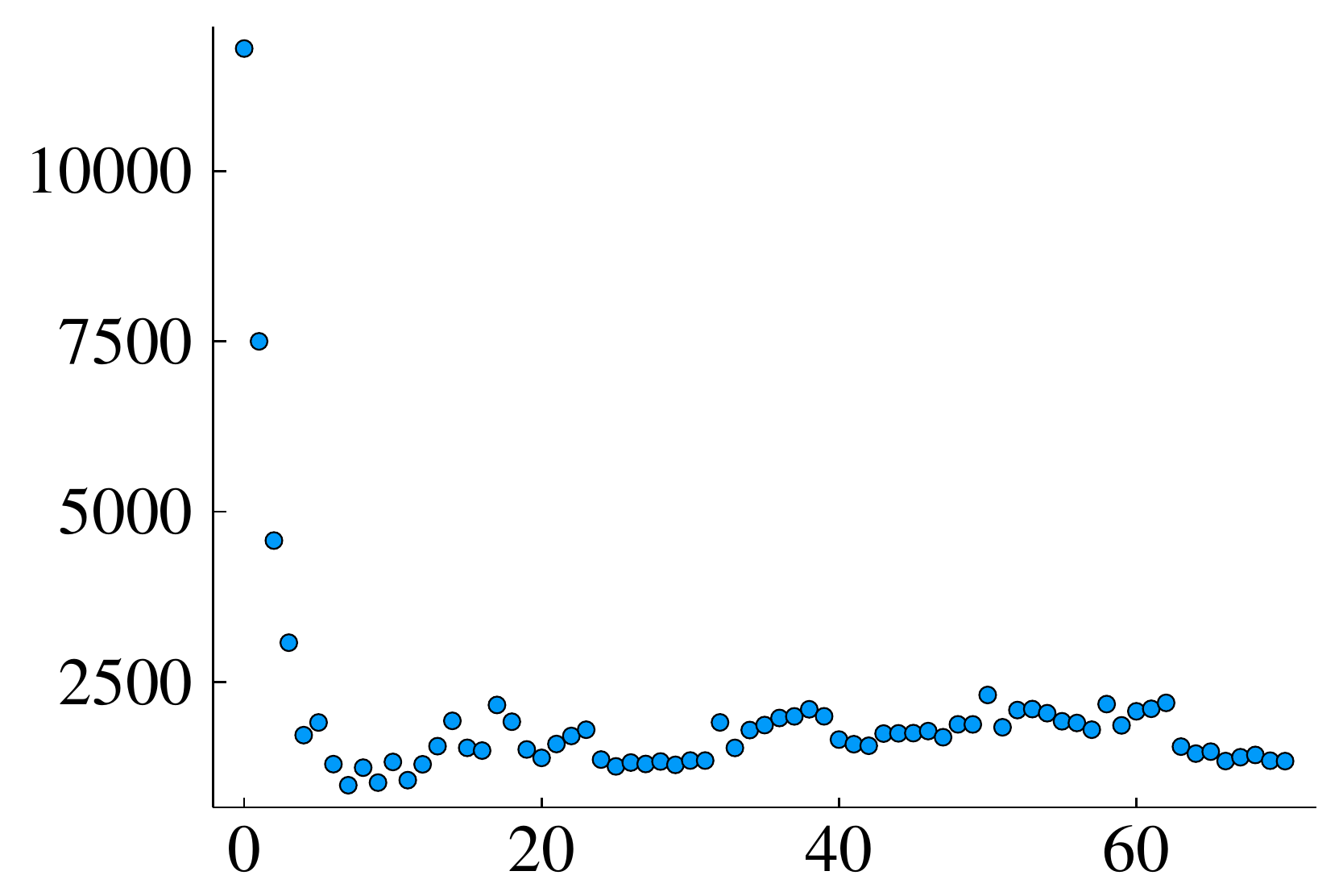}};
        \node[text=black](xlabel) at (0.05\linewidth,-.33\linewidth) {{\footnotesize memory size $m$}};
        \node[rotate=90, text=black] (ylabel) at (-.51\linewidth,0) {{\footnotesize $\times 10^3$ scaled iteration}};
    \end{tikzpicture}
    \end{subfigure}
\caption{Number of scaled iterations until the normalized $M$-induced distance to the solution gets smaller than some value \emph{tol} vs.\ memory size with the settings in the experiments of \cref{fig:CP_Alg4_m_breastCancer} (\emph{tol}$=10^{-8}$, \emph{top left} panel), \cref{fig:CP_Alg4_m_sonar} (\emph{tol}$=10^{-6}$, \emph{top right} panel), and \cref{fig:CP_Alg4_m_colonCancer} (\emph{tol}$=10^{-4}$, \emph{bottom} panel); using \cref{alg:PD-DWIFOB}, where $m=0$ corresponds to the Chambolle--Pock method and $m=1$ corresponds to the inertial primal--dual method of \cite{nofob-increments}.}
\label{fig:scaled_itr_vs_m}
\end{figure}

\begin{figure}
    \centering
    \begin{subfigure}{0.49\textwidth}
    \centering
    \begin{tikzpicture}[scale=1.0]
      \node[inner sep=0pt] (fig) at (0,0) {\includegraphics[ width = 0.92\linewidth,keepaspectratio]{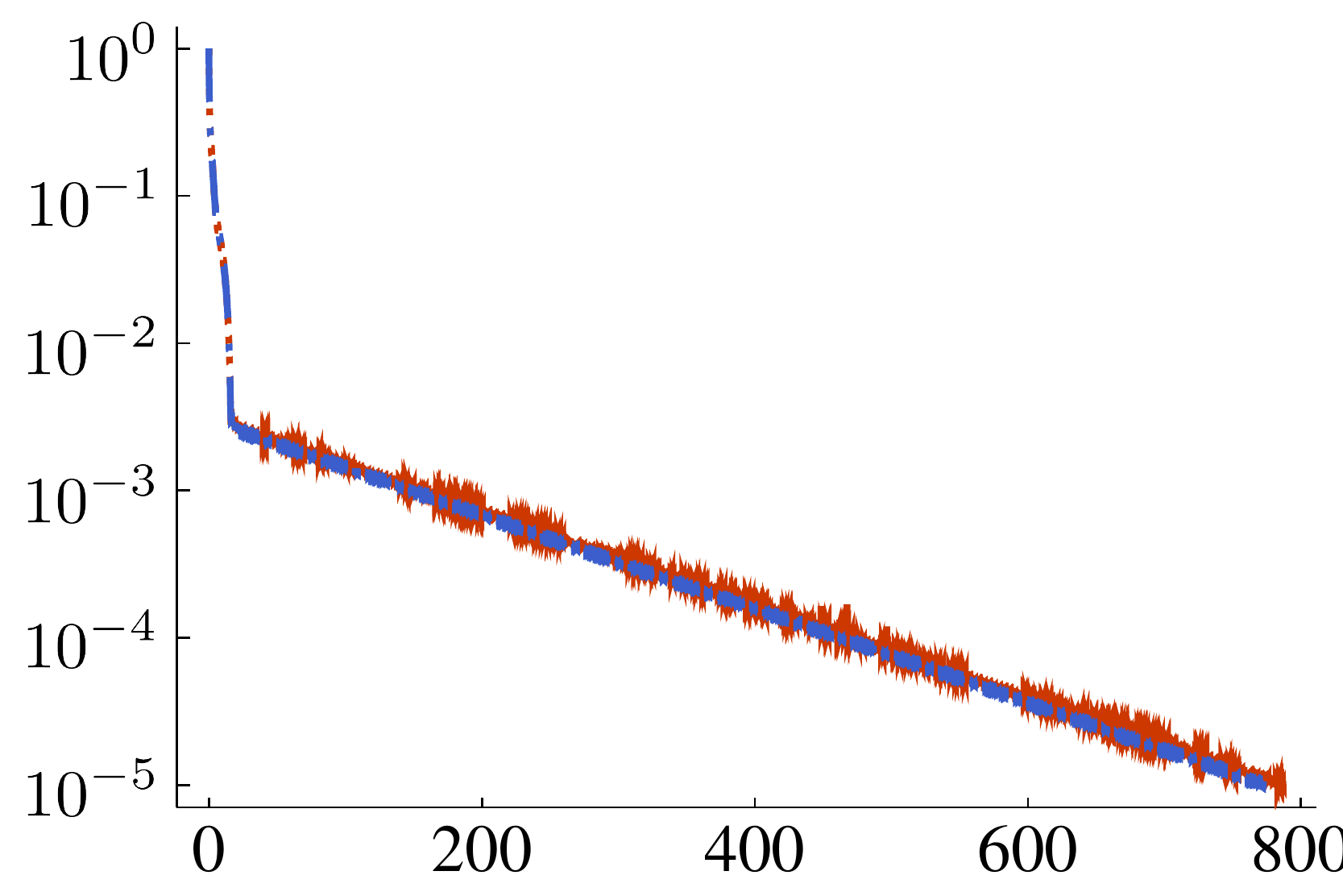}};
        \node[text=black](xlabel) at (0.05\linewidth,-.33\linewidth) {{\footnotesize $\times 10^3$ iteration}};
        \node[rotate=90, text=black] (ylabel) at (-.51\linewidth,0) {$\frac{\norm{(x_n,\mu_n)-(x^\star,\mu^\star)}_M}{\norm{(x_0,\mu_0)-(x^\star,\mu^\star)}_M}$};
    \end{tikzpicture}
    \end{subfigure}
    \begin{subfigure}{0.49\textwidth}
    \centering
    \begin{tikzpicture}[scale=1.0]
        \node[inner sep=0pt] (fig) at (0,0) {\includegraphics[ width = 0.92\linewidth,keepaspectratio]{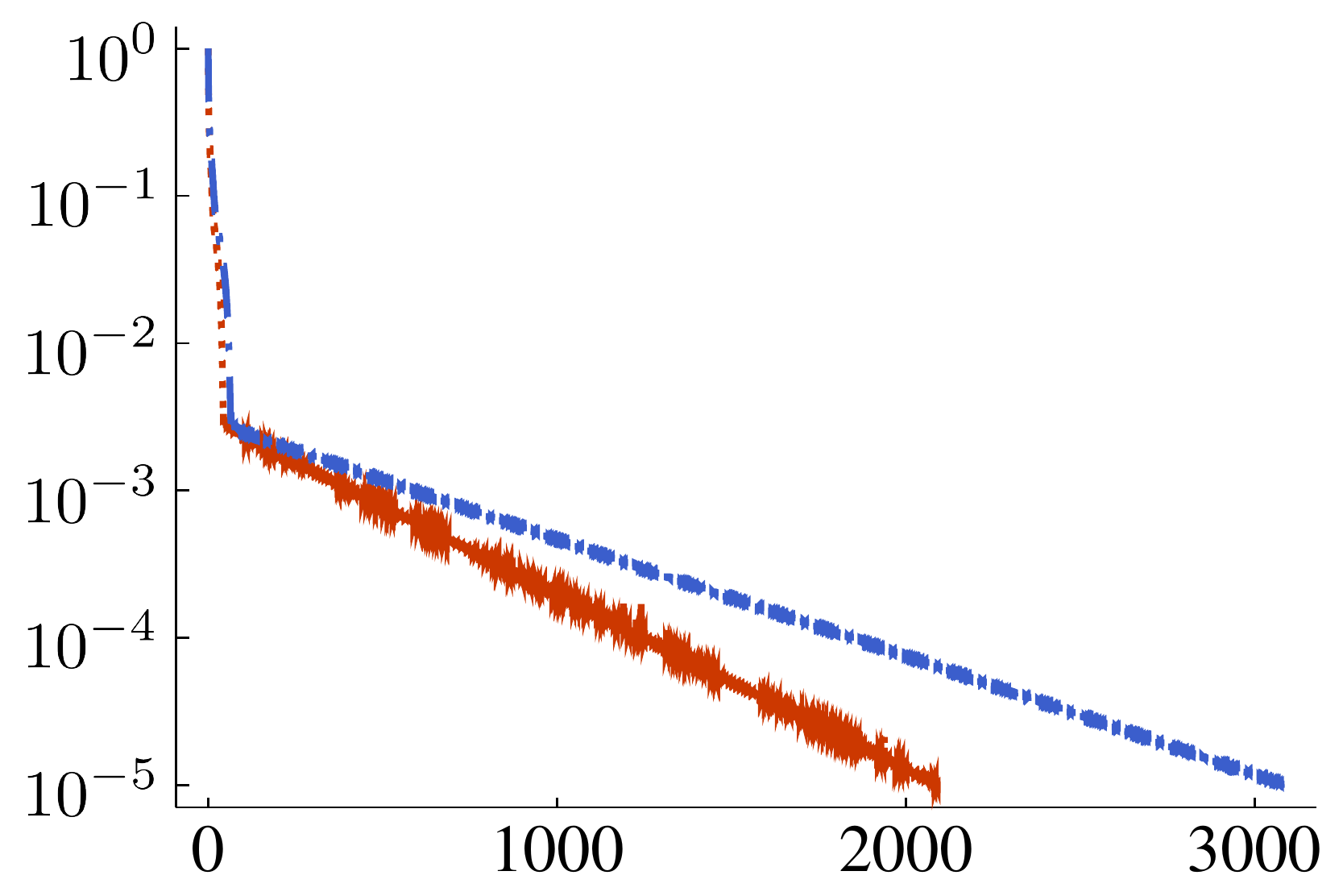}};
        \node[text=black](xlabel) at (0.05\linewidth,-.33\linewidth) {{\footnotesize $\times 10^3$ scaled iteration}};
    \end{tikzpicture}
    \end{subfigure}
    \begin{subfigure}{0.49\textwidth}
    \centering
    \begin{tikzpicture}[scale=1.0]
        \node[inner sep=0pt] (fig) at (0,0) {\includegraphics[ width = 0.92\linewidth,keepaspectratio]{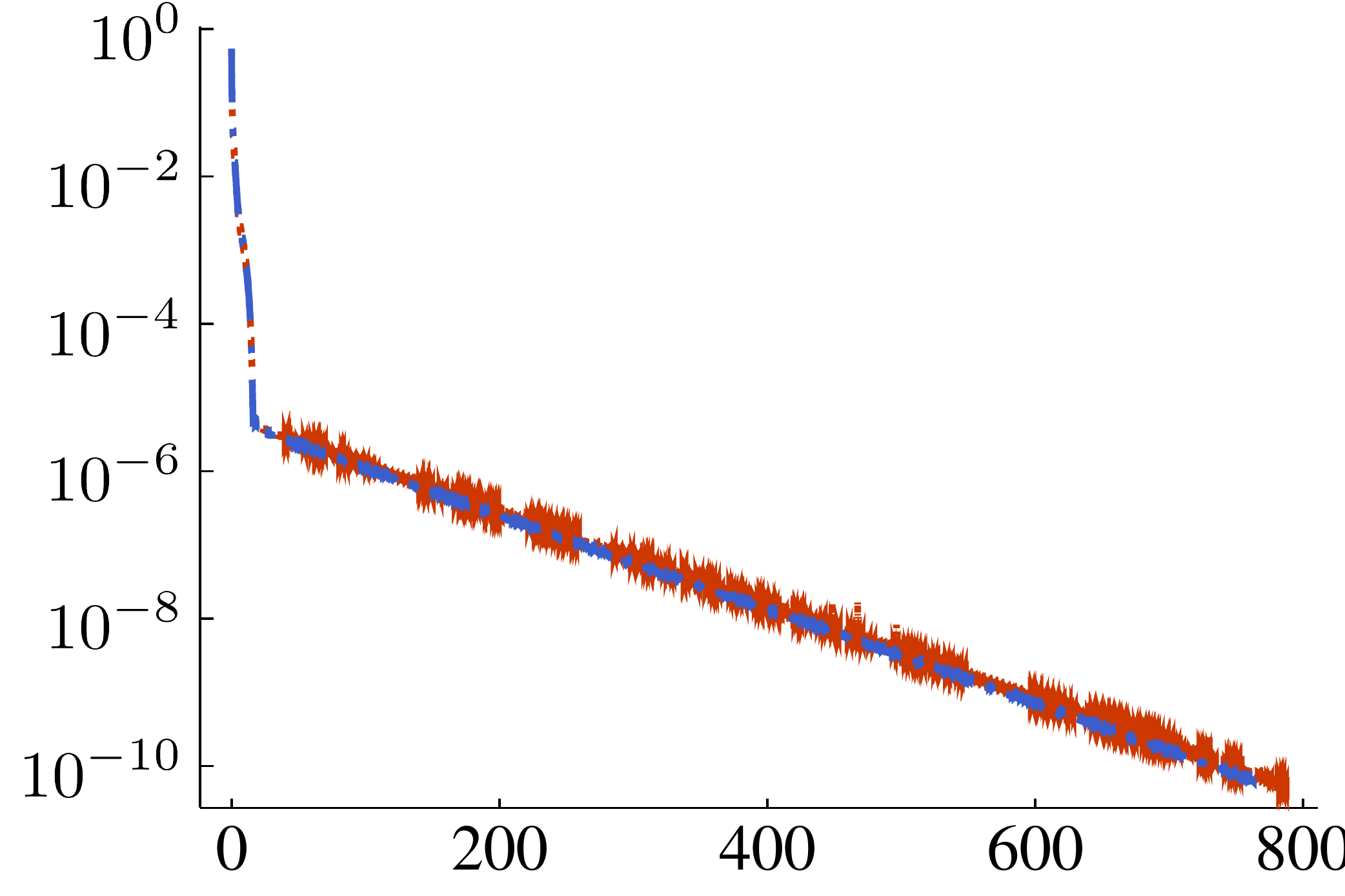}};
        \node[text=black](xlabel) at (0.05\linewidth,-.33\linewidth) {{\footnotesize $\times 10^3$ iteration}};
        \node[rotate=90, text=black] (ylabel) at (-.51\linewidth,0) {$V_n$};
    \end{tikzpicture}
    \end{subfigure}
    \begin{subfigure}{\textwidth}
    \centering
    \begin{tikzpicture}[scale=1]
        \matrix[draw=black, very thin]  at (-4,0) { %\matrix [draw=black, very thin]
            \draw[redd, dotted, very thick](0,0)--(0.54,0);&\node {{\scriptsize with recursive evaluation of $L$, $L^*$, and $\norm{\cdot}_M$}};\\
            \draw[bluee, dashdotted, very thick](0,0)--(0.54,0);&\node {{\scriptsize with direct evaluation of $L$, $L^*$, and $\norm{\cdot}_M$~~~~~}};\\
        };
    \end{tikzpicture}
    \end{subfigure}
\caption{Comparing the impact of \emph{recursive} and \emph{direct} evaluation of $L$, $L^*$, and $\norm{\cdot}_M$ on the convergence pattern of Alg\labelcref{alg:PD-DWIFOB} ($\lambda=1.0$, $m=25$, $\xi=10^{-6}$) for problem \eqref{eq:l1_reg_svm} with $\delta = 0.1$, on the \emph{colon cancer} dataset \cite{chang2011libsvm}; \emph{Top panels:} normalized $M$-induced distance to the solution vs.\ iteration number (\emph{top left}) and scaled iteration number (\emph{top right}); \emph{bottom panel:} $V_n$ (defined in \eqref{eq:decreasing-quantity})  vs.\ iteration number.}
\label{fig:alg4_spikes}
\end{figure}
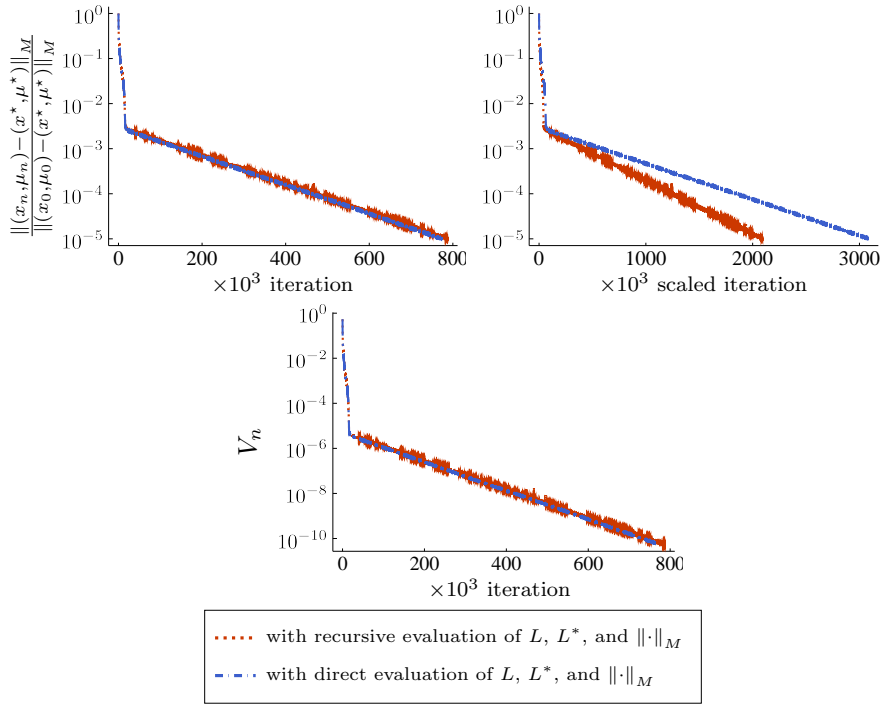

\begin{figure}
    \centering
    \begin{subfigure}{\textwidth}
    \centering
    \begin{tikzpicture}[scale=1]
        \node[text=black](title) at (0.05\linewidth,0.33\linewidth) {{\scriptsize \hspace{3mm} CP vs.\ Alg\labelcref{alg:PD-DWIFOB} ($\lambda=1.0,m,\xi=10^{-5}$) \hspace{15mm} CP vs.\ RAA ($m,\xi=10^{-5}$)}};
    \end{tikzpicture}
    \end{subfigure}
    \begin{subfigure}{0.49\textwidth}
    \centering
    \begin{tikzpicture}[scale=1.0]
      \node[inner sep=0pt] (fig) at (0,0) {\includegraphics[ width = 0.92\linewidth,keepaspectratio]{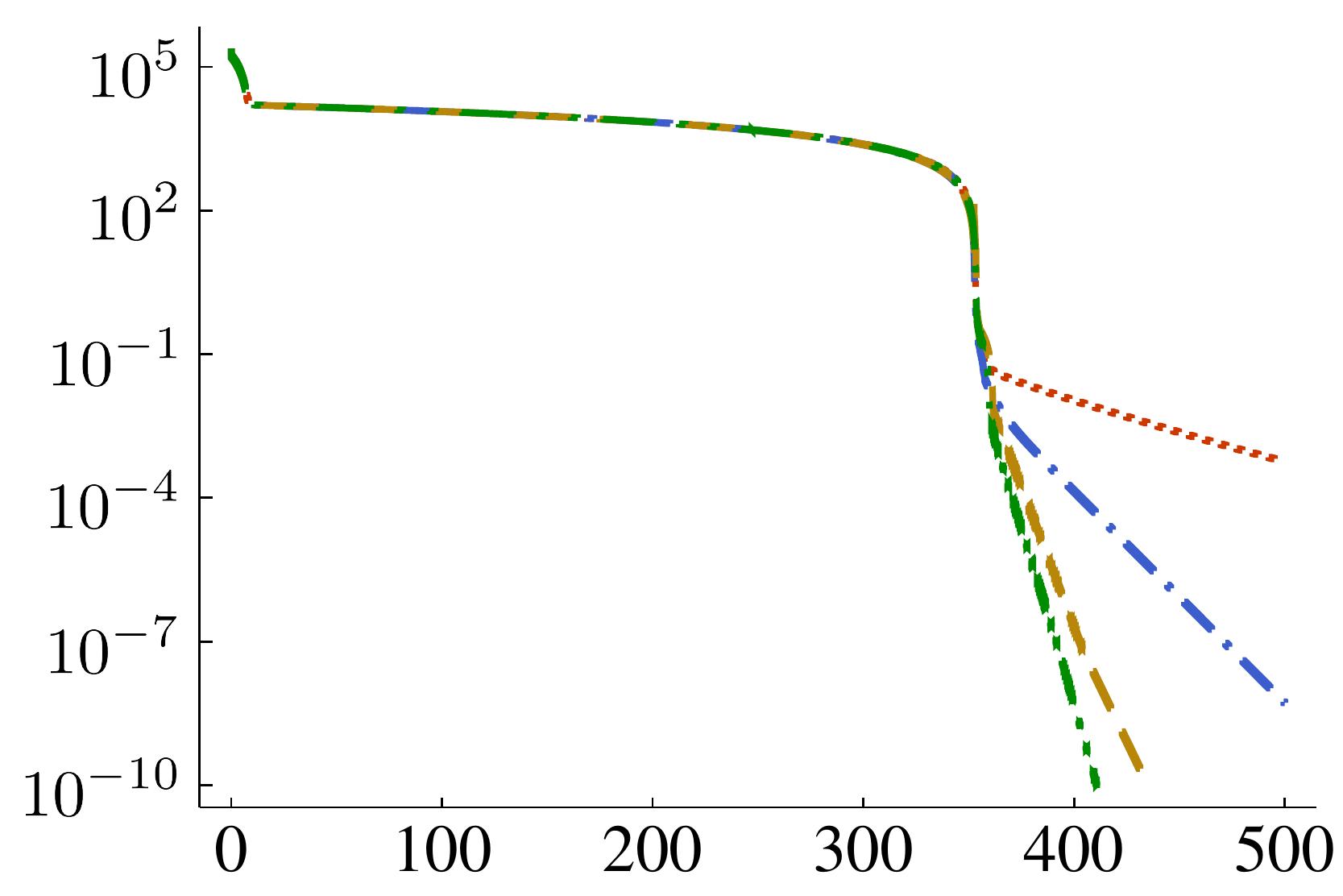}};
        \node[text=black](xlabel) at (0.05\linewidth,-.33\linewidth) {{\footnotesize $\times 10^3$ iteration}};
        \node[rotate=90, text=black] (ylabel) at (-.51\linewidth,0) {{\footnotesize $\norm{(x_n,\mu_n)-(x^\star,\mu^\star)}_M$}};
    \end{tikzpicture}
    \end{subfigure}
    \begin{subfigure}{0.49\textwidth}
    \centering
    \begin{tikzpicture}[scale=1.0]
        \node[inner sep=0pt] (fig) at (0,0) {\includegraphics[ width = 0.92\linewidth,keepaspectratio]{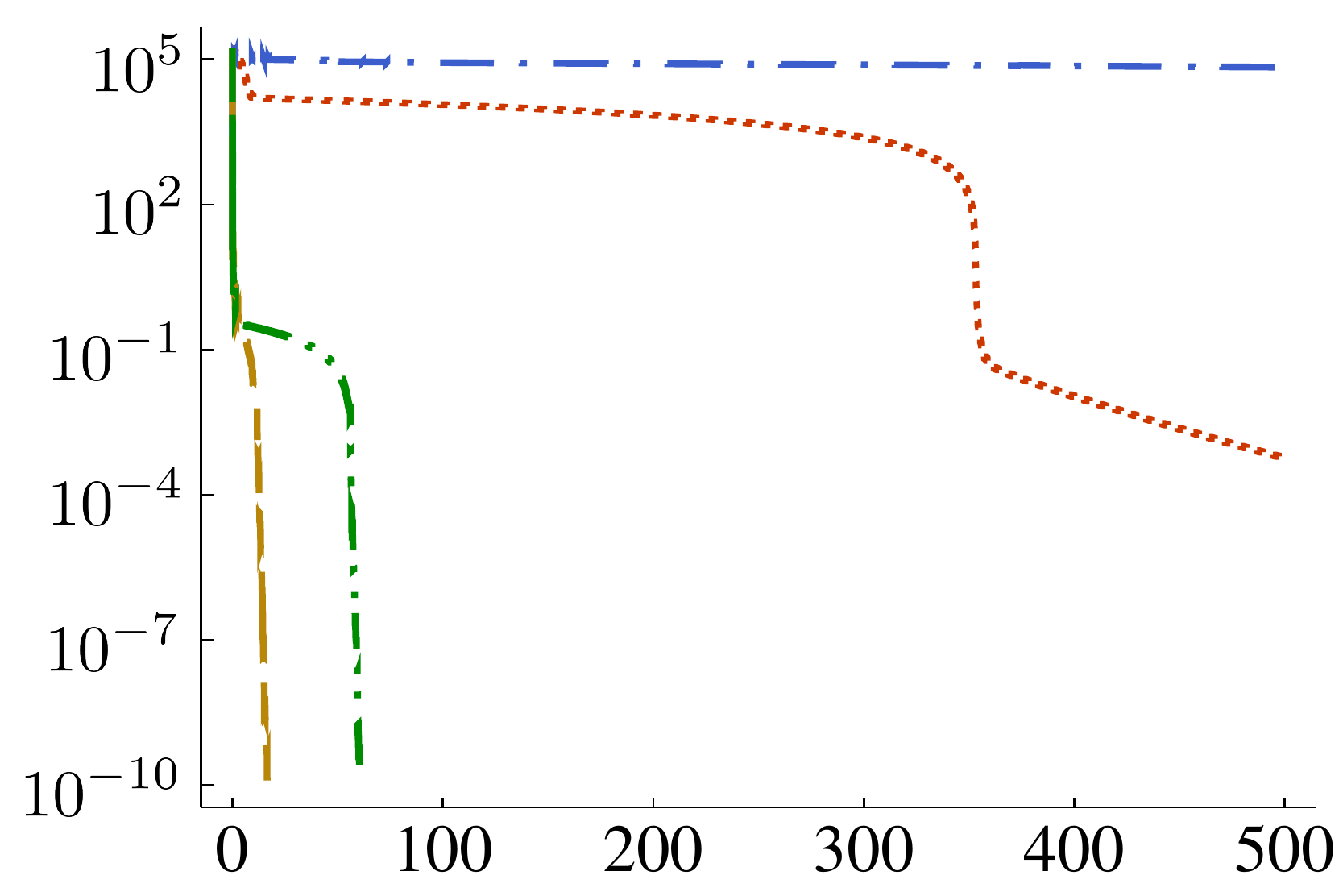}};
        \node[text=black](xlabel) at (0.05\linewidth,-.33\linewidth) {{\footnotesize $\times 10^3$ iteration}};
    \end{tikzpicture}
    \end{subfigure}
    \begin{subfigure}{\textwidth}
    \centering
    \begin{tikzpicture}[scale=1]
        \node[text=black](title) at (0.05\linewidth,0.33\linewidth) {{\scriptsize \hspace{3mm} CP vs.\ Alg\labelcref{alg:PD-DWIFOB} ($\lambda=1.0,m,\xi=10^{-6}$) \hspace{15mm} CP vs.\ RAA ($m,\xi=10^{-6}$)}};
    \end{tikzpicture}
    \end{subfigure}
    \begin{subfigure}{0.49\textwidth}
    \centering
    \begin{tikzpicture}[scale=1.0]
      \node[inner sep=0pt] (fig) at (0,0) {\includegraphics[ width = 0.92\linewidth,keepaspectratio]{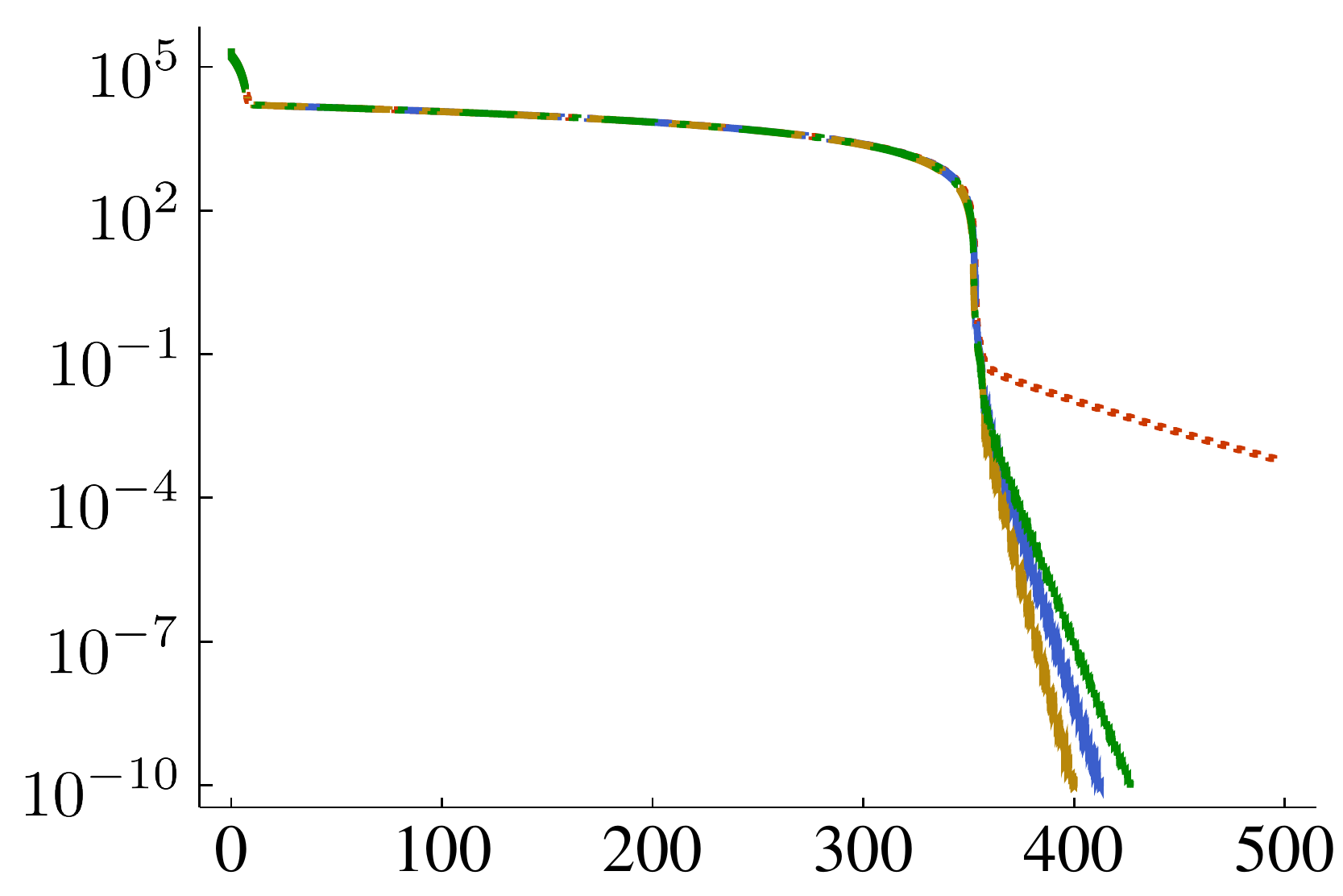}};
        \node[text=black](xlabel) at (0.05\linewidth,-.33\linewidth) {{\footnotesize $\times 10^3$ iteration}};
        \node[rotate=90, text=black] (ylabel) at (-.51\linewidth,0) {{\footnotesize $\norm{(x_n,\mu_n)-(x^\star,\mu^\star)}_M$}};
    \end{tikzpicture}
    \end{subfigure}
    \begin{subfigure}{0.49\textwidth}
    \centering
    \begin{tikzpicture}[scale=1.0]
        \node[inner sep=0pt] (fig) at (0,0) {\includegraphics[ width = 0.92\linewidth,keepaspectratio]{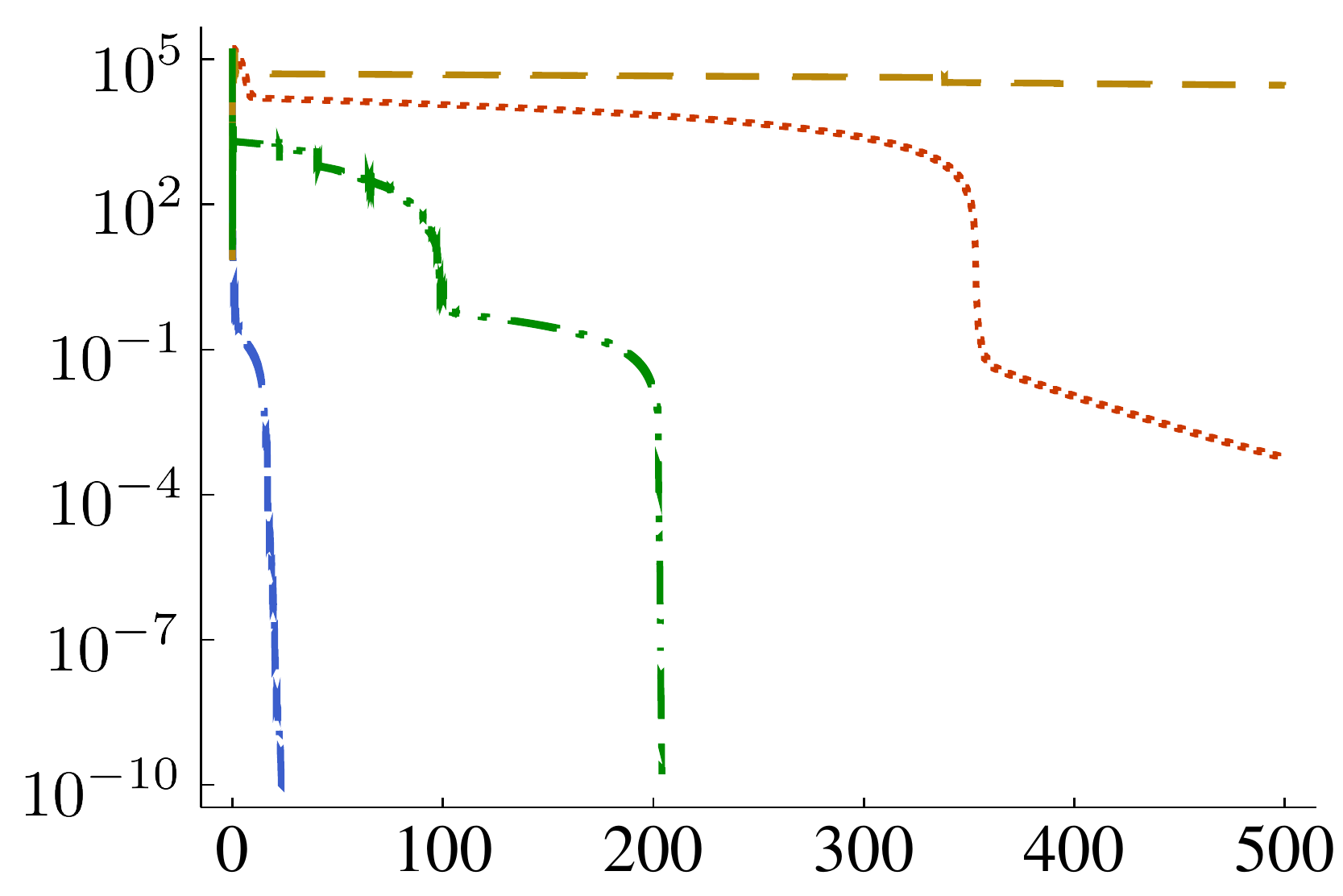}};
        \node[text=black](xlabel) at (0.05\linewidth,-.33\linewidth) {{\footnotesize $\times 10^3$ iteration}};
    \end{tikzpicture}
    \end{subfigure}
    \begin{subfigure}{\textwidth}
    \centering
    \begin{tikzpicture}[scale=1]
        \node[text=black](title) at (0.05\linewidth,0.33\linewidth) {{\scriptsize \hspace{3mm} CP vs.\ Alg\labelcref{alg:PD-DWIFOB} ($\lambda=1.0,m,\xi=10^{-7}$) \hspace{15mm} CP vs.\ RAA ($m,\xi=10^{-7}$)}};
    \end{tikzpicture}
    \end{subfigure}
    \begin{subfigure}{0.49\textwidth}
    \centering
    \begin{tikzpicture}[scale=1.0]
      \node[inner sep=0pt] (fig) at (0,0) {\includegraphics[ width = 0.92\linewidth,keepaspectratio]{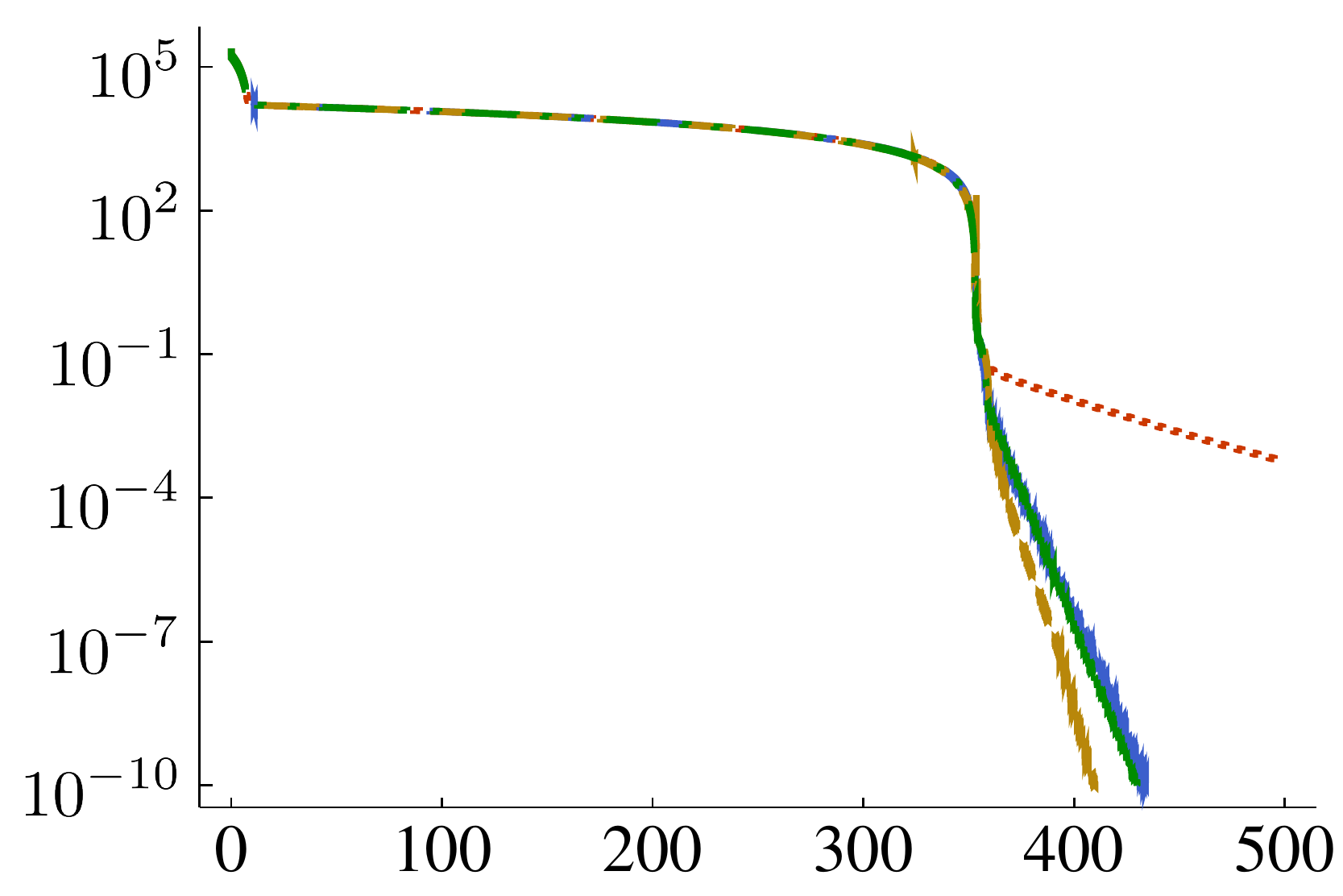}};
        \node[text=black](xlabel) at (0.05\linewidth,-.33\linewidth) {{\footnotesize $\times 10^3$ iteration}};
        \node[rotate=90, text=black] (ylabel) at (-.51\linewidth,0) {{\footnotesize $\norm{(x_n,\mu_n)-(x^\star,\mu^\star)}_M$}};
    \end{tikzpicture}
    \end{subfigure}
    \begin{subfigure}{0.49\textwidth}
    \centering
    \begin{tikzpicture}[scale=1.0]
        \node[inner sep=0pt] (fig) at (0,0) {\includegraphics[ width = 0.92\linewidth,keepaspectratio]{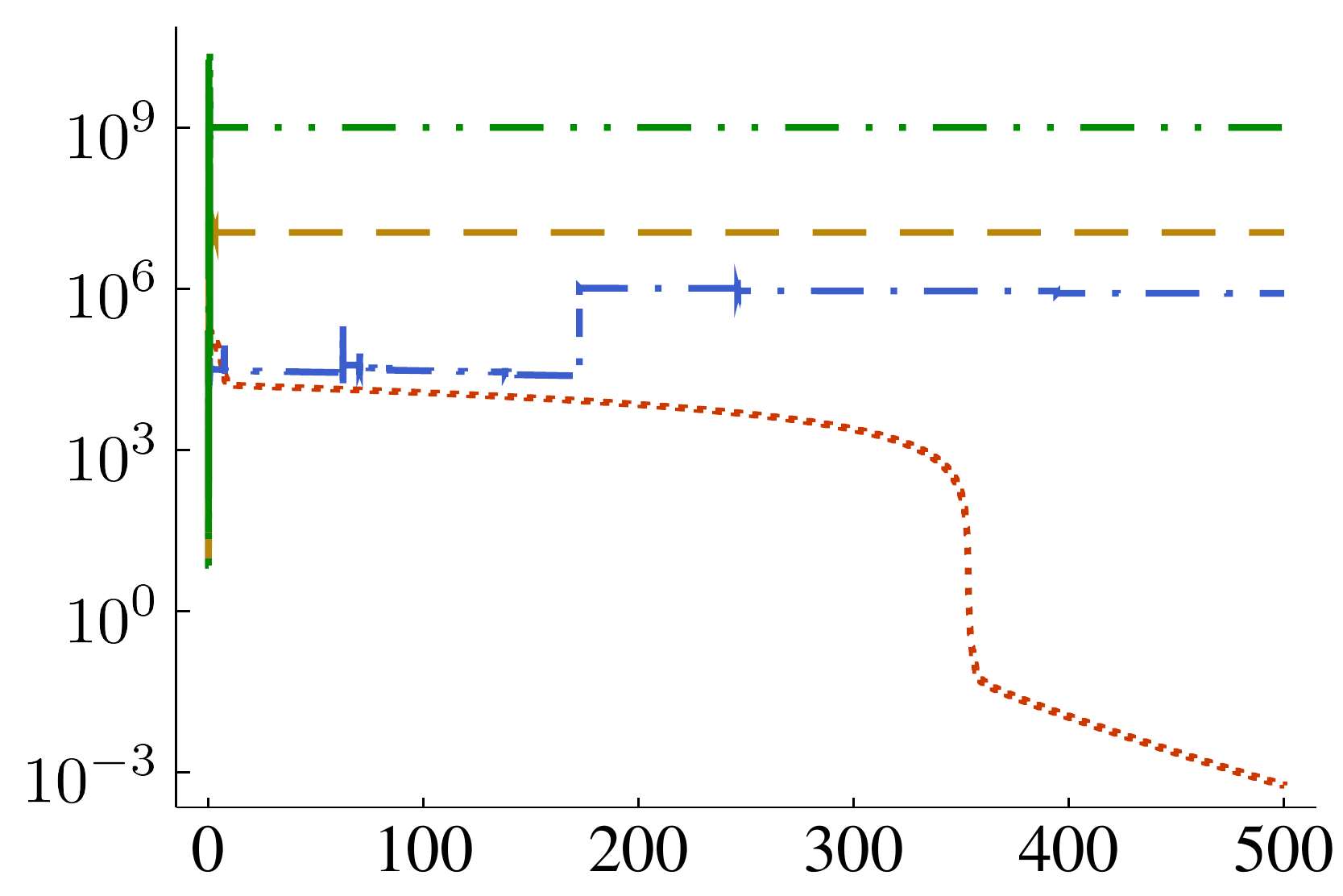}};
        \node[text=black](xlabel) at (0.05\linewidth,-.33\linewidth) {{\footnotesize $\times 10^3$ iteration}};
    \end{tikzpicture}
    \end{subfigure}
    \begin{subfigure}{\textwidth}
    \centering
    \begin{tikzpicture}[scale=1]
        \matrix[draw=black, very thin]  at (-4,0) { %\matrix [draw=black, very thin]
            \draw[redd, dotted, very thick](0,0)--(0.54,0);&\node {{\scriptsize CP}};&
            \node {~~~~~};&
            \draw[bluee, dashdotted, very thick](0,0)--(0.5,0);&\node {{\scriptsize $m=5$}};&
            \node {~~~~~};&
            \draw[brownn, dashed, very thick](0,0)--(0.5,0);&\node {{\scriptsize $m=10$}};&
            \node {~~~~~};&
            \draw[greenn, dashdotdotted, very thick](0,0)--(0.5,0);&\node {{\scriptsize $m=15$}};\\
        };
    \end{tikzpicture}
    \end{subfigure}
\caption{Normalized $M$-induced distance to the solution vs.\ iteration number for the $l_1$-norm regularized SVM problem \eqref{eq:l1_reg_svm} with $\delta = 0.5$ on the \emph{breast cancer} dataset \cite{chang2011libsvm} with 683 samples and 10 features. Solved using the Chambolle--Pock algorithm, Alg\labelcref{alg:PD-DWIFOB} ($\lambda=1.0$, $m$, $\xi$) (\emph{left-hand side} plots), and RAA ($m$, $\xi$) (\emph{right-hand side} plots) for several memory sizes and Tikhonov regularization parameters, all with $\tau=\sigma = 0.99/\norm{L}$. In this case, the initial point is set far from the origin, namely, at a distance of approximately $2.6\times10^5$ to the origin.}
\label{fig:Alg4_RAA_breastCancer}
\end{figure}

\Crefrange{fig:CP_Alg4_m_breastCancer}{fig:CP_Alg4_m_colonCancer} provide a comparison between the Chambolle--Pock method and \cref{alg:PD-DWIFOB} for several memory size values using different datasets. The figures show that for the considered different values of the memory size $m$, \cref{alg:PD-DWIFOB} outperforms the Chambolle--Pock method. It can also be seen that increasing the memory size $m$ in \cref{alg:PD-DWIFOB} improves the local convergence rate. However, by increasing $m$ in \cref{alg:PD-DWIFOB}, the computational cost of solving the least-squares problem increases, while the computational cost of the resolvent steps is fixed. Therefore, it is expected that there is an optimal memory size beyond which increasing $m$ degrades the performance (compared to the optimal one). This can be better seen in \cref{fig:scaled_itr_vs_m}, which shows the number of scaled iterations until the $M$-scaled distance of $(x_n,\mu_n)$ to the solution is less than some value \emph{tol}, against the memory size. It is seen  that we get good performance for a wide range of memory sizes (typically $10\leq m \leq25$). It is also good to mention that even for small or large $m$, we still see a considerable improvement compared to the Chambolle--Pock method.

\Cref{fig:alg4_spikes} shows the impact of using direct evaluation  of $L$, $L^*$, and $\norm{\cdot}_M$ instead of the proposed recursive method of \cref{subsec:efficient_evaluation}, on the convergence pattern of \cref{alg:PD-DWIFOB}. The experiment is done with the same setting as in the one reported in \cref{fig:CP_Alg4_m_colonCancer} for the case of $m=25$. The top right plot shows that the suggested method of recursive evaluation of \cref{alg:PD-DWIFOB} considerably decreases the overall computational cost, in this instance by about $30\%$. Additionally, it is observed that by using the suggested recursive evaluation of $L$, $L^*$, and $\norm{\cdot}_M$, we might see some unexpected spikes in the plots, which are caused by accumulated errors due to recursive evaluations, while using the direct evaluation method does not result in such spikes. The bottom plot in \cref{fig:alg4_spikes} compares 
\begin{equation}\label{eq:decreasing-quantity}
\begin{aligned}
    V_{n} &\defeq \normsq{\begin{bmatrix}x_{n+1}\\\mu_{n+1}\end{bmatrix}-\begin{bmatrix}\xs\\\mu^\star\end{bmatrix}}_M\\
    &\qquad\qquad+ {\lambda_n(2-\lambda_n)}\norm{\begin{bmatrix}p_{x,n}\\p_{\mu,n}\end{bmatrix}-\begin{bmatrix}x_{n}\\\mu_{n}\end{bmatrix}+\frac{\lambda_n-1}{2-\lambda_n}\begin{bmatrix}u_{x,n}\\u_{\mu,n}\end{bmatrix}}_M^2
\end{aligned}
\end{equation}
for the case of direct and recursive evaluation methods. According to \cite[Lemma 1]{nofob-increments} with exact evaluation of $L$, $L^*$, and $M$, this quantity should be decreasing, which is confirmed by the figure. However, this is not the case for the recursive evaluation method due to accumulated errors.

The results of experiments with the Chambolle--Pock method, \cref{alg:PD-DWIFOB}, and RAA are shown in \cref{fig:Alg4_RAA_breastCancer}. The plots on the left-hand side compare the Chambolle--Pock algorithm and \cref{alg:PD-DWIFOB} and the plots on the right-hand side show the convergence of RAA versus the Chambolle--Pock algorithm. For these experiments, the algorithms are initialized far from the origin (at $(x_0,\mu_0)=10^4\times\mathbf{1}_{694}$, where $\mathbf{1}_{694}$ is a vector of ones with $694$ elements). We see that RAA is not globally convergent; however, when it converges, it does so fast. It is also seen that RAA is really sensitive to parameter variations; and besides that, for it to perform well, there should be a reasonable match between the regularization parameter and its memory size (see the middle plot of RAA). On the other hand, \cref{alg:PD-DWIFOB} is more robust against variations in parameters. These results suggest that \cref{alg:PD-DWIFOB} is more reliable than RAA in the sense of robustness against variations in parameters and also predictability of its behavior.

The distances to a solution for RAA that do not converge to zero in \cref{fig:Alg4_RAA_breastCancer} have not converged although they seem to have flat asymptotes. In fact, consecutive iterates differ a lot and the primal iterate inserted into the objective function \eqref{eq:l1_reg_svm_reform} gives values that are several orders of magnitude larger than the optimal value, also at the end of the simulation. This rules out that the algorithm converges to a different solution (if it exists) than all the other methods do.

%%%%%%%%%%%%%%%%%%%%%%%%%%%%%%%%%%%
%%%%%%%%%%%%%%%%%%%%%%%%%%%%%%%%%%%
\section{Conclusion}\label{sec:conclusion}
We have proposed a novel scheme to solve structured monotone inclusion problems. By combining a variant of FB splitting with deviations with an extrapolation technique similar to that of Anderson acceleration, we introduced the \dwifob{} algorithm. Using the flexibility that the FB algorithm with deviations provides, we introduced a primal--dual variant of the \dwifob{} algorithm. Numerical experiments on an $l_1$-norm regularized SVM problem showed that the primal--dual variant of the \dwifob{} algorithm outperforms the Chambolle--Pock primal--dual method. Additionally, we compared the performance of the primal--dual variation of \dwifob{} to the regularized Anderson acceleration  on the same benchmark problem. The results showed that, in addition to only being locally (though fast) convergent, Anderson acceleration is very sensitive to the variations in choice of parameters while primal--dual \dwifob{} is much more robust against them. This makes the behavior of the \dwifob{} algorithm more reliable and predictable.

\paragraph{Acknowledgement.} The authors would like to thank Bo Bernhardsson (Department of Automatic Control,  Lund University) for his valuable feedback on this work. This research was partially supported by Wallenberg AI, Autonomous Systems and Software Program (WASP) funded by the Knut and Alice Wallenberg Foundation.  Sebastian Banert was partially supported by ELLIIT.

\printbibliography

\end{document}